\documentclass[3p,longnamesfirst,sort&compress]{elsarticle}

\usepackage{amsmath,amsthm,amssymb,mathrsfs}
\usepackage{thmtools,thm-restate}
\usepackage{enumerate}
\usepackage{slashed}
\usepackage{txfonts,dsfont}
\usepackage[breaklinks,colorlinks]{hyperref}
\usepackage{nameref}
\usepackage{cleveref}

\makeatletter
\newcommand{\autorefcheckize}[1]{%
\expandafter\let\csname @@\string#1\endcsname#1%
\expandafter\DeclareRobustCommand\csname relax\string#1\endcsname[1]{%
\csname @@\string#1\endcsname{##1}\wrtusdrf{##1}}%
\expandafter\let\expandafter#1\csname relax\string#1\endcsname
}
\makeatother

\declaretheorem[numberwithin=section]{theorem}
\declaretheorem[sibling=theorem, name=Lemma]{lem}
\declaretheorem[sibling=theorem, name=Corollary]{cor}

\declaretheorem[sibling=theorem, name=Remark]{rem}
\declaretheorem[numberwithin=section, name=Definition]{defn}

\declaretheorem[name=Problem]{prob}

\declaretheorem[name=Theorem]{main}

\numberwithin{equation}{section}

\newcommand{\norm}[1]{\left\lVert#1\right\rVert}
\newcommand{\abs}[1]{\left\lvert#1\right\rvert}
\newcommand{\set}[1]{\left\{#1\right\}}
\newcommand{\hin}[2]{\left\langle#1,#2\right\rangle}

\newcommand*{\Rmn}[1]{\uppercase\expandafter{\romannumeral#1}}
\newcommand*{\dif}{\mathop{}\!\mathrm{d}}

\DeclareMathOperator{\Div}{div}
\DeclareMathOperator{\trace}{tr}

\journal{XXX}

\allowdisplaybreaks

\begin{document}

\begin{frontmatter}
\title{Critical quasilinear equations on Riemannian manifolds\tnoteref{sw}}

\author[xtu]{Linlin Sun}
\ead{sunll@xtu.edu.cn}

\author[gzhu,amss,ucas]{Youde Wang\corref{wyd}}
\ead{wyd@math.ac.cn}

\address[xtu]{School of Mathematics and Computational Science, Xiangtan University, Xiangtan 411105, China}

\address[gzhu]{School of Mathematics and Information Sciences, Guangzhou University, Guangzhou 510006, China}
\address[amss]{Hua Loo-Keng Key Laboratory Mathematics, Institute of Mathematics, Academy of Mathematics and Systems Science, Chinese Academy of Sciences, Beijing 100190, China}
\address[ucas]{School of Mathematical Sciences, University of Chinese Academy of Sciences, Beijing 100049, China}

\cortext[wyd]{Corresponding author.}
\tnotetext[sw]{Y. Wang is supported by National Natural Science Foundation of China (Grant No. 12431003).}

\begin{abstract}
In this paper, we investigate critical quasilinear elliptic partial differential equations on a complete Riemannian manifold with nonnegative Ricci curvature. By exploiting a new and sharp nonlinear Kato inequality and establishing some Cheng-Yau type gradient estimates, we classify positive solutions to the critical $p$-Laplace equation and show rigidity concerning the ambient manifold. Our results extend and improve some previous conclusions in the literature. Similar results are obtained for solutions to the quasilinear Liouville equation involving the $n$-Laplace operator, where $n$ corresponds to the dimension of the ambient manifold.
\end{abstract}


\end{frontmatter}


\section{Introduction}
In the last half century, significant progress has been made in the study of the Lane-Emden equation given by
\begin{align}\label{eq:lane-emden}
\Delta u + u^\alpha=0, \quad\mbox{in}~~ M^n,
\end{align}
where $\left(M^n, g\right)$ is an $n$-dimensional complete Riemannian manifold and $\Delta$ is the Laplace-Beltrami operator with respect to the metric $g$ (see \cite{GidSpr81global, Ding-Ni85, CafGidSpr89asymptotic, WanWei23on} and references therein). In the case $-\infty<\alpha<\frac{n+2}{n-2}$, it has been shown that this equation does not admit any positive solution if the underlying manifold is a noncompact complete manifold with nonnegative Ricci curvature (see \cite{GidSpr81global, WanWei23on}). On the contrary, by using ``finite domain approximation" Ding and Ni in \cite{Ding-Ni85} proved the following:
\begin{quotation}
For any $c>0$, the equation $\Delta u + au^\alpha=0$ in $\mathbb{R}^n$, where $\alpha\geq\frac{n+2}{n-2}$ and $a$ is a positive constant, possesses a positive solution $v$ with $\norm{v}_{ L^\infty\left(\mathbb{R}^n\right)} = c$.
\end{quotation}
Ding \cite{Ding86infini} also proved that the Lane-Emden equation \eqref{eq:lane-emden} in $\mathbb{R}^n$ with $\alpha = \frac{n+2}{n-2}$ admits infinitely many sign-changing solutions.

On the other hand, one has also studied the following quasilinear Lane-Emden equation
\begin{align}\label{eq:quasi-lane-emden}
 \Delta_p u + u^\alpha=0, \quad\mbox{in}~~ M^n
\end{align}
where the $p$-Laplacian $\Delta_p$, acting on functions $u$, is defined as follows:
\begin{equation*}
\Delta_pu = \Div\left(\abs{\nabla u}^{p-2}\nabla u\right),
\end{equation*}
since the $p$-Laplace operator plays a crucial role in various fields including differential geometry, partial differential equations (PDEs), physics, and stochastic models. For instance, Serrin and Zou \cite[Corollary \Rmn{2}]{SerZou02cauchy} proved that the above classical quasilinear Lane-Emden equation \eqref{eq:quasi-lane-emden} admits no positive solution in $\mathbb{R}^n$ if
\begin{align*}
0<\alpha <\dfrac{(n+1)p-n}{n-p}\quad \text{and} \quad 1<p<n.
\end{align*}
Actually, the above Liouville theorem established in \cite{GidSpr81global, SerZou02cauchy} has been extended to the cases
\begin{align*}
-\infty <\alpha <\dfrac{(n+1)p-n}{n-p} \quad\text{and}\quad 1<p<n,
\end{align*}
for details we refer to \cite{HeWanWei24gradient}. Very recently, J. He and the authors of this paper \cite{HeSunWan24optimal} proved the optimal Liouville theorem on the quasilinear Lane-Emden equation \eqref{eq:quasi-lane-emden} on a noncompact complete Riemannian manifold with nonnegative Ricci curvature. More concretely, they proved that the quasilinear Lane-Emden equation \eqref{eq:quasi-lane-emden}, defined on a complete Riemannian manifold with nonnegative Ricci curvature, does not admit any positive solution if $\alpha\in \left(-\infty,\, p_S-1\right)$ where $p_S$ is the Sobolev critical exponent given by
$$p_S\equiv\dfrac{np}{(n-p)^+},$$
where $p_S = \frac{np}{n-p}$ for $1 < p < n$ is associated with the Sobolev embedding $W^{1,p}_0\left(\Omega\right) \subset L^{p_{S}}\left(\Omega\right)$ for any bounded domain $\Omega \subset \mathbb{R}^n$. Moreover, they also established a universal $\log$-gradient estimate for positive solutions to this equation on a complete manifold with Ricci curvature bounded from below.

It is important to note that for the nonlinear heat equation
$$ u_t -\Delta u = u^\alpha\quad \mbox{in}\,\,\mathbb{R}^n,\quad \alpha>1,$$
Quittner \cite{Quit21op} addressed the longstanding conjecture regarding the nonexistence of positive classical solutions within the subcritical range $\alpha(n-2)<n+2$.

On the contrary, if $\alpha = p_S-1$ in the above Lane-Emden equation \eqref{eq:quasi-lane-emden}, we obtain the critical $p$-Laplace equation in $\mathbb{R}^n$
$$\Delta_p u + u^{p_s-1}=0$$
which is related to the extremals in the Sobolev inequality. It is well-known that it does admit the following solution family:
\begin{align}\label{eq:solution-p}
u(x)=\left(a+b\abs{x-x_0}^{\frac{p}{p-1}}\right)^{-\frac{n-p}{p}},\quad\quad n\left(\frac{n-p}{p-1}\right)^{p-1}ab^{p-1}=1,
\end{align}
where $x_0\in\mathbb{R}^n$, $a>0$ and $b>0$. Moreover, One has also shown that the above solutions, which are usually called as \emph{Aubin-Talenti bubbles}, are the only $\mathcal{D}^{1,p}$-solutions to this equation (for $1< p <n$), where
$$\mathcal{D}^{1,p}\left(\mathbb{R}^n\right)\coloneqq\left\{v\in L^{p_S}\left(\mathbb{R}^n\right): \int_{\mathbb{R}^n}\abs{\nabla v}^p <+\infty\right\}.$$
The result can also be regarded as a Liouville theorem corresponding to the critical quasilinear equation. For further details we refer to \cite{Ding86infini, LiLi05yamabe, Sci16classification, Vet16priori} and references therein. The Liouville theorems for subcritical or critical semilinear elliptic equations in the Heisenberg group has also been considered, we refer the interested reader to \cite{MaOu23liouville, CatLiMonRon24on}.

The Liouville equation in $\mathbb{R}^2$ reads
$$\Delta u + e^{2u}=0$$
which is a close relative of the critical $p$-Laplace equation. A family of standard solutions to the Liouville equation in $\mathbb{R}^2$ are expressed as
\begin{align}\label{eq:liouville-2}
u(x) = -\ln\left(a + b \abs{x - x_0}^2\right)
\end{align}
where $x_0 \in \mathbb{R}^n$ and the parameters $a$ and $b$ are positive constants fulfilling the condition $4ab = 1$. Solutions $u$ to this equation in $\mathbb{R}^2$ with $e^{2u} \in L^1\left(\mathbb{R}^2\right)$ were classified by Chen and Li \cite{CheLi91classification}, who showed that $u$ must be of the above form \eqref{eq:liouville-2}. One also provided a complete classification with respect to asymptotic behavior, stability and intersections
properties of radial smooth solutions to the equation $-\Delta u = e^u$ on Riemannian model manifolds $\left(M^n, g\right)$ in dimension $n \geq 2$ (see \cite{BerFerGanRoy2023}).

\subsection{Motivations and geometric background}

In this paper we are focus on the \emph{critical $p$-Laplace equation} ($1<p<n$)
\begin{equation}\label{eq:critical-p}
-\Delta_pu = u^{p_S - 1}
\end{equation}
and the quasilinear \emph{Liouville equation}
\begin{equation}\label{eq:critical-n}
-\Delta_nu = e^{nu}
\end{equation}
in Euclidean space $\mathbb{R}^n$ or a complete Riemannian manifold $\left(M^n, g\right)$ with nonnegative Ricci curvature. Although the equation \eqref{eq:critical-p} in $\mathbb{R}^n$ has been extensively studied by numerous mathematicians, an interesting and challenging problem is remained, i.e.,
\begin{prob}\label{prob:1}
Whether or not any positive solution to \eqref{eq:critical-p} in $\mathbb{R}^n$ must be an \emph{Aubin-Talenti bubble}?
\end{prob}

Naturally, one want to extend some conclusions on the equations \eqref{eq:critical-p} in $\mathbb{R}^n$ to the case the domain $\left(M^n, g\right)$ is a complete Riemannian manifold, for instance, extending the above Liouville theorems for the critical Laplace equation ($p=2$) to the general Riemannian setting. However, the problem turns out to be a challenging and still open problem:
\begin{prob}\label{prob:2}
The problem is divided into three parts:
\begin{enumerate}[(A)]
\item
In fact, there is a \textbf{conjecture} on the following strong rigidity result: let $\left(M^n, g\right)$ be a complete noncompact Riemannian manifold with nonnegative Ricci curvature and let $u \in C^2\left(M^n\right)$ be a solution to
$$\Delta u + u^{\frac{n+2}{n-2}}=0, \quad u \geq 0,\,\, \mbox{in}\,\, M^n,$$
then either $u \equiv 0$ in $M^n$ or $\left(M^n, g\right)$ is isometric to $\mathbb{R}^n$ with the Euclidean metric and $u$ is an \emph{Aubin-Talenti bubble} (see \cite{CatMonRonWan24on} posted on arXiv).

\item
Naturally, one wonder whether or not the following strong rigidity result for critical $p$-Laplace ($1<p<n$) equation holds: let $\left(M^n, g\right)$ be a complete noncompact Riemannian manifold with nonnegative Ricci curvature and let $u$ be a nonnegative and weak solution to the critical $p$-Laplace equation \eqref{eq:critical-p}, then either $u \equiv 0$ in $M^n$ or $\left(M^n, g\right)$ is isometric to $\mathbb{R}^n$ with the Euclidean metric and $u$ is an \emph{Aubin-Talenti bubble}.

\item A weak version of the above (B), whether or not such manifold admits any Sobolev minimizer of the best Sobolev constant denoted by
$\mathcal{S}_p\left(M^n\right)$ for each $1<p<n$? In this context, the best Sobolev constant, denoted as $\mathcal{S}_p\left(M^n\right)$, is expressed by the following equation:
\begin{align*}
 \mathcal{S}_p\left(M^n\right) = \sup_{0\neq u\in C^{\infty}_0\left(M^n\right)} \frac{\norm{u}_{L^{p_S}\left(M^n\right)}}{\norm{\nabla u}_{L^p\left(M^n\right)}},
\end{align*}
while a Sobolev minimizer $u$ for this best Sobolev constant $\mathcal{S}_p\left(M^n\right)$ is characterized by $u\in \mathcal{D}^{1,p}(M)$ and satisfies
$$\norm{u}_{L^{p_S}\left(M^n\right)} = \mathcal{S}_p\left(M^n\right)\norm{\nabla u}_{L^p\left(M^n\right)}.$$
\end{enumerate}
\end{prob}

For the equation \eqref{eq:critical-n} we may ask the following problem, i.e.,
\begin{prob}\label{prob:3}
Let $\left(M^n, g\right)$ be a complete noncompact Riemannian manifold and with nonnegative Ricci curvature and let $u$ be a weak solution to the Liouville equation \eqref{eq:critical-n}.
What conditions on $u$ can enforce that $M^n$ is isometric to the Euclidean space $\mathbb{R}^n$ and the solution $u$ is given by
\begin{align*}
u(x) = -\ln\left(a + b \abs{x - x_0}^{\frac{n}{n-1}}\right)\quad\mbox{with}\quad \frac{n^n}{(n-1)^{n-1}}ab^{n-1} = 1,
\end{align*}
where $x_0\in\mathbb{R}^n$ and the parameters $a$ and $b$ are positive constants?
\end{prob}

The goal of the present paper is to explore the classification of solutions to critical quasilinear elliptic equations on a complete Riemannian manifold with nonnegative Ricci curvature, and hence to answer partially the above problems. Furthermore, we will examine the influence of these solutions on the rigidity properties of the domain manifold. It is well known that such issues are crucial in many applications such as a priori estimates, blow-up analysis, and asymptotic analysis.
\medskip

The equations \eqref{eq:critical-p} and \eqref{eq:critical-n} are closely related to conformal geometry. On the one hand, the well-known Yamabe equation (cf. \cite{LeePar87yamabe, Wei19on}), a fundamental problem in conformal geometry, seeks to determine the existence of a conformal metric with constant scalar curvature on any closed Riemannian manifold. Specifically, let $\left(M^n, g\right)$ be a closed Riemannian manifold of dimension $n$. For $n \geq 3$, the question arises: can we find a conformal metric $\tilde{g} = u^{\frac{4}{n-2}} g$ such that the scalar curvature of $\tilde{g}$ remains constant? Consequently, the Yamabe problem is equivalent to solving the following partial differential equation:
\begin{align*}
-\Delta_g u + \dfrac{n-2}{4(n-1)} R_g u = \lambda u^{\frac{n+2}{n-2}},
\end{align*}
where $R_g$ denotes the scalar curvature associated with the initial metric $g$ and $\lambda$ represents a constant. On the other hand, the Nirenberg problem asks to identify functions $K$ on the two-sphere $\mathbb{S}^2$ for which there exists a metric $\tilde{g}$ on $\mathbb{S}^2$ conformal to the standard metric $g$ such that $K$ is the Gaussian curvature of $\tilde{g}$. The question becomes: can we find a conformal metric $\tilde{g} = e^{2u} g$ such that $K$ is the Gaussian curvature of $\tilde{g}$? That is, this problem reduces to studying the following nonlinear equation:
\begin{align*}
-\Delta_{\mathbb{S}^2} u + 1 = K e^{2u}.
\end{align*}
Besides, the equation \eqref{eq:critical-n} appears in the blow-up analysis on Moser-Trudinger inequality (see \cite{Li05Moser}) and other geometric problems (see \cite{ChengNi91, Gel63, MaQing21har, WeiXu99Clas, Xu05Uni}).

\vspace{2ex}

Now, let us recall some fundamental facts and results which are related to the equation \eqref{eq:critical-p}. Firstly we recall the variational nature of the critical $p$-Laplace equation: the action functional associated to \eqref{eq:critical-p} is given by
$$ E_M(u)\coloneqq \dfrac{1}{p}\int_M \abs{\nabla u}^p - \dfrac{1}{p_S}\int_M \abs{u}^{p_S},$$
indeed it is well-known that the Euler-Lagrange equation associated to this action functional is
\begin{align*}
-\Delta_pu=\abs{u}^{p_s-2}u.
\end{align*}
We know that the critical Sobolev embedding $W^{1,p}\left(\mathbb{R}^n\right) \subset L^{p_S}\left(\mathbb{R}^n\right)$ asserts the existence of a positive constant $\mathcal{S}_p\left(\mathbb{R}^n\right)$ such that for every function $v \in C_0^\infty\left(\mathbb{R}^n\right)$,
\begin{equation}\label{eq:sobolev}
\norm{v}_{L^{p_S}\left(\mathbb{R}^n\right)} \leq \mathcal{S}_p\left(\mathbb{R}^n\right)\norm{
\nabla v}_{L^p\left(\mathbb{R}^n\right)}.
\end{equation}
It is easy to see that the extremal functions of the above Sobolev inequality satisfy \eqref{eq:critical-p} if they exist.

The identification of the best constant $\mathcal{S}_p\left(\mathbb{R}^n\right)$ for $1 < p < n$ dates back to Aubin \cite{Aub76problems} when $p=2$ and Talenti \cite{Tal76best} for general $p$. To obtain the best constant, Aubin and Talenti employed the P\'olya-Szeg\"o inequality by utilizing Schwarz symmetrization and the sharp isoperimetric inequality in $\mathbb{R}^n$, thus reducing a multidimensional problem to a more tractable one-dimensional case. The best constant $\mathcal{S}_p\left(\mathbb{R}^n\right)$ in this inequality \eqref{eq:sobolev} when $p \in (1,\, n)$ is explicitly known, and the extremal functions are just given by the \emph{Aubin-Talenti bubbles}.

In the limit case when $p = 1$, we have
\begin{equation*}
\lim_{p \to 1} \mathcal{S}_p\left(\mathbb{R}^n\right) = n^{-1} \omega_n^{-1/n},
\end{equation*}
where $\omega_n$ is the volume of the unit ball in $\mathbb{R}^n$, while \eqref{eq:sobolev} reduces to the sharp $L^1$-Sobolev inequality, which is equivalent to the sharp isoperimetric inequality in $\mathbb{R}^n$. A genuinely new proof for the sharp Sobolev inequality \eqref{eq:sobolev} has been obtained by Cordero-Erausquin, Nazaret, and Villani \cite{CorNazVil04mass} based on mass transportation theory and the Brenier map \cite{Bre91polar}.

Significant efforts have recently been devoted to the study of optimal Sobolev inequalities on Riemannian manifolds, as surveyed in \cite{DruHeb02AB} and its references. Recently, the sharp Sobolev inequality on complete noncompact Riemannian manifolds with nonnegative Ricci curvature was obtained by Brendle \cite{Bre23sobolev} using the method of the Alexandrov-Bakelman-Pucci. This work addresses an open question posed by Cordero-Erausquin, Nazaret, and Villani \cite{CorNazVil04mass} concerning noncompact, complete Riemannian manifolds with nonnegative Ricci curvature. Balogh and Krist\'aly \cite{Bal-Kris} gave an alternative proof of the rigidity result of Brendle and shew that the Sobolev constant is sharp. Concretely, the main result is stated as follows:

\begin{quotation}
Let $\left(M^n, g\right)$ be a noncompact, complete $n$-dimensional Riemannian manifold with nonnegative Ricci curvature and Euclidean volume growth, i.e., $0 < \mathrm{AVR}_g \leq 1$. If $p \in [1, n)$, then for all $v \in C_0^\infty\left(M^n\right)$, it holds that
\begin{align*}
\norm{v}_{L^{p_{S}}\left(M^n\right)} \leq \mathcal{S}_p\left(\mathbb{R}^n\right)\mathrm{AVR}_g^{-1/n}\norm{\nabla v}_{L^p\left(M^n\right)}.
\end{align*}
Furthermore, the constant $\mathcal{S}_p\left(\mathbb{R}^n\right) \mathrm{AVR}_g^{-1/n}$ is shown to be sharp.
\end{quotation}
Here, $\mathrm{AVR}_g$ denotes the asymptotic volume ratio of $\left(M^n, g\right)$, defined as
\begin{align*}
\mathrm{AVR}_g = \lim_{R \to \infty} \dfrac{\mathrm{Vol}(B_R(o))}{\omega_n R^n},
\end{align*}
which is independent of the base point $o$ due to the Bishop-Gromov volume comparison theorem when the Ricci curvature is nonnegative, the constant $\omega_n = \pi^{n/2}/\Gamma(1+ n/2)$ plays a normalization role and it is the volume of the Euclidean unit ball in $\mathbb{R}^n$.

Notably, Krist\'aly \cite{Kri24sharp} has also provided a symmetrization-free, fully $L^2$-optimal transport (OT)-based proof for the sharp $L^p$-Sobolev inequality for $1\leq p< n$ in the case of $n$-dimensional Riemannian manifolds with nonnegative Ricci curvature.
\vspace{2ex}

For \autoref{prob:1} associated with the critical $p$-Laplace equation \eqref{eq:critical-p}, i.e., ``whether any positive solution to \eqref{eq:critical-p} in $\mathbb{R}^n$ must be of the form \eqref{eq:solution-p} or not", we recall some previous results in the literature.

In case of $p=2$, this problem was solved by Caffarelli, Gidas and Spruck \cite{CafGidSpr89asymptotic}. They proved that any positive solution to \eqref{eq:critical-p} in $\mathbb{R}^n$ is radial and hence given by \eqref{eq:solution-p}. Essential tools in their proof are the moving planes method and the Kelvin transform, which allow reducing the problem to the study of solutions with nice decaying properties at infinity. For the quasilinear case, i.e., $p \neq 2$, the Kelvin transform is not available, making it more complicated than the semilinear case. Nevertheless, the method of moving planes has been exploited by Damascelli, Merchant, Montoro and Sciunzi \cite{DamMerMonSci14radial} for $\frac{2n}{n+2} \leq p < 2$, Sciunzi \cite{Sci16classification} for $2 < p < n$, and V\'etois \cite{Vet16priori} for $1 < p \leq 2$, under the additional assumption of finite energy. That is, they proved that
\begin{quote}
Each positive solution $u$ to \eqref{eq:critical-p} in $\mathbb{R}^n$ with $u \in L^{p_{S}}\left(\mathbb{R}^n\right)$ and $
\nabla u \in L^p\left(\mathbb{R}^n\right)$ must be of the form \eqref{eq:solution-p}.
\end{quote}
Such classification results were extended by Ciraolo, Figalli and Roncoroni \cite{CirFigRon20symmetry} to the anisotropic setting and in convex cones of $\mathbb{R}^n$ using the method of integration by parts, which is related to Gidas and Spruck \cite{GidSpr81global} and Serrin and Zou \cite{SerZou02cauchy}. The same method has also been successfully used in analogous problems, see \cite{HeSunWan24optimal, LinMa24liouville} for example.
\medskip

But for the case $p\neq 2$ and without any further assumptions, the problem is still completely open, until recently Catino, Monticelli and Roncoroni \cite{CatMonRon23on} gave a result in dimensions $n = 2,\, 3$ for $\frac{n}{2} < p < 2$. This result was extended by Ou \cite{Ou22on} for a large range $p \in \left[\frac{n+1}{3},\, n\right)$. More recently, after a refined analysis, V\'etois \cite{Vet24note} improved upon a similar result obtained by Ou to the range $p \in (\tilde{p}_n,\, n)$ where
$$\tilde{p}_n \in \left(\dfrac{n}{3},\, \dfrac{n+1}{3}\right)\quad \mbox{and}\quad \tilde{p}_n \sim \dfrac{n}{3} + \dfrac{1}{n}\quad \mbox{as}\,\, n \to \infty.$$

On the other hand, the critical $p$-Laplace equation \eqref{eq:critical-p} in complete, connected and noncompact Riemannian manifolds $M^n$ of dimension $n$ with nonnegative Ricci curvature has recently been studied by some mathematicians. For instance, Fogagnolo, Malchiodi and Mazzieri \cite{FogMalMaz23note, FogMalMaz24correction}, Catino and Monticelli \cite{CatMon24semilinear}, and Ciraolo, Farina and Polvara \cite{CirFarPol24classification} explored the semilinear case. Later, Catino, Monticelli and Roncoroni \cite{CatMonRon23on} studied the quasilinear case and obtained several classification results under additional geometric and analytic assumptions.

In particular, in \cite{CatMon24semilinear}, the authors proved that the only positive classical solutions to \eqref{eq:critical-p} are given by the Aubin-Talenti bubbles \eqref{eq:solution-p} with $p=2$, provided that the Riemannian manifold is isometric to the Euclidean space, $n = 3$, or $u$ has finite energy or satisfies suitable conditions at infinity. These results were extended by Ciraolo, Farina and Polvara \cite{CirFarPol24classification}. Actually, the conjecture given in \autoref{prob:2} associated with \eqref{eq:critical-p} is known to hold when $p=2$ and the dimension of the Riemannian manifold is $n =3,\,4,\,5$ (see \cite{CatMon24semilinear, CirFarPol24classification}), otherwise, up to now it can not be proved except the solutions fulfill some additional assumptions.
\medskip

Moreover, as an immediate consequence of the conclusion (Theorem 1.2 in \cite{CatMon24semilinear}) Catino and Monticelli proved:
\begin{quote}
A complete noncompact nonflat $n$-dimensional ($n\geq 3$) Riemannian manifold with nonnegative Ricci curvature does not admit any Sobolev minimizer of $\mathcal{S}_2\left(M^n\right)$.
\end{quote}
An alternative proof of the above result can be recovered using a recent result by Brendle \cite{Bre23sobolev} and the results by the second named author of this paper \cite{Wan94on}, Ledoux \cite{Led99} and Xia \cite{Xia}, for further details we refer to \cite{CatMon24semilinear}.

It should be mentioned that Muratori and Soave \cite{MuSo23cartan} studied recently some PDEs on Cartan-Hadamard manifolds, which are related to Sobolev inequalities and critical $p$-Laplace equations on Cartan-Hadamard manifolds, and established some rigidity results for these related PDEs (also see \cite{BerFerGri2014, BonGaGri2013}).
\medskip

As regards the Liouville equation \eqref{eq:critical-n}, a standard solution to the Liouville equation \eqref{eq:critical-n} in $\mathbb{R}^n$ is expressed as
\begin{align}
\label{eq:solution-n}
u(x) = -\ln\left(a + b \abs{x - x_0}^{\frac{n}{n-1}}\right)
\end{align}
where $x_0 \in \mathbb{R}^n$ and the parameters $a$ and $b$ are positive, fulfilling the condition
$$\dfrac{n^n}{(n-1)^{n-1}}ab^{n-1} = 1.$$
As mentioned above, in the case $n = 2$, solutions $u$ to \eqref{eq:critical-n} in $\mathbb{R}^2$ with $e^{2u} \in L^1\left(\mathbb{R}^2\right)$ were classified by Chen and Li \cite{CheLi91classification}, who showed that $u$ must be of the form \eqref{eq:solution-n}. Their method relies on Ding's Lemma, the moving plane method, and a previous result by Brezis and Merle \cite{BreMer91uniform} on the upper boundedness of solutions. Moreover, in \cite{Espo18non} Esposito classified completely entire solutions of the general case in $\mathbb{R}^n$ with finite mass, i.e., solutions $u$ to \eqref{eq:critical-n} in $\mathbb{R}^n$ with $e^{nu}\in L^1\left(\mathbb{R}^n\right)$ must be of the above form \eqref{eq:solution-n}.

\medskip

Meanwhile, the Liouville equation on complete, connected, and noncompact Riemannian surfaces with nonnegative Gaussian curvature has also been studied by Catino and Monticelli \cite{CatMon24semilinear}, Cai and Lai \cite{CaiLai24liouville}, and Ciraolo, Farina, and Polvara \cite{CirFarPol24classification}. Notably, Cai and Lai \cite{CaiLai24liouville} provided the following classification result.
\begin{quotation}
Let $u$ be a solution to \eqref{eq:critical-n} on a complete, connected and noncompacted Riemannian surface $\Sigma$ with nonnegative Gaussian curvature. Let $r(x)$ be the distance function on $\Sigma$ with respect to a fixed point. Assume $e^{2u}\in L^1\left(\Sigma\right)$ and $u(x)\geq-2\ln r(x)-o\left(\ln r(x)\right)$ for $r(x)$ large, then $\Sigma$ must be isometric to the Euclidean plane $\mathbb{R}^2$. Moreover, $-2$ is optimal in the sense that there exists a nonflat $\Sigma$ which admits solutions satisfying $u(x)\sim \gamma\ln r(x)$ for any $\gamma<-2$.
\end{quotation}
It is noteworthy that Catino and Monticelli \cite{CatMon24semilinear} achieved a comparable result by imposing a more stringent condition on $u$. Similarly, Ciraolo, Farina, and Polvara \cite{CirFarPol24classification} reached a parallel conclusion, forgoing the requirement of finite total Gaussian curvature but instead enforcing more rigorous conditions on the asymptotic properties of the solution.
\medskip

It is natural to expect stronger classification results for the critical $p$-Laplace equation \eqref{eq:critical-p} and the Liouville equation \eqref{eq:critical-n} on Riemannian manifolds with nonnegative Ricci curvature. No doubt, extending such classification results for the critical equations to more general Riemannian manifolds requires the introduction of different techniques than those used in the Euclidean setting, which strongly rely on the conformal invariance of the problem and the rich structure of the conformal group of the ambient space.

\subsection{Main results}

Now, we are in the position to state our main results on the critical $p$-Laplace equation \eqref{eq:critical-p} and the Liouville equation \eqref{eq:critical-n}.

Firstly, we investigate the critical $p$-Laplace equation \eqref{eq:critical-p} within the framework of an $n$-dimensional Riemannian manifold $M^n$.

\begin{defn}A function $u\in W^{1,p}_{loc}\left(M^n\right)\cap L^{\infty}_{loc}\left(M^n\right)$ is said to be a weak solution to \eqref{eq:critical-p} if
\begin{align*}
\int \abs{\nabla u}^{p-2}\hin{\nabla u}{\nabla\phi}=\int u^{p_S-1}\phi,\quad\forall \phi\in C_0^{\infty}\left(M^n\right).
\end{align*}
\end{defn}
The first main result is the following
\begin{main}\label{thm:main1}
Let $M^n$ be a complete, connected and noncompact Riemannian manifold of dimension $n$ and with nonnegative Ricci curvature. Let $u$ be a positive and weak solution to the critical $p$-Laplace equation \eqref{eq:critical-p} in $M^n$, i.e.,
\begin{align*}
-\Delta_pu=u^{\frac{np}{n-p}-1}.
\end{align*}
Assume one of the following conditions holds
\begin{enumerate}[(A)]
\item $p_n<p<n$ where
\begin{align*}
p_n=\begin{cases}
\frac{n^2}{3n-2}, &n\in\set{2,3,4},\\
\frac{n^2+2}{3n}, &n\geq5.
\end{cases}
\end{align*}
\item $1<p<n$ and
\begin{align*}
\limsup_{R\to\infty} R^{q-n}\int_{B_{R}(o)}u^{\frac{pq}{n-p}}<\infty
\end{align*}
for some fixed point $o\in M^n$ and some constant
$$q>\dfrac{(n^2-3n+1)p+n}{(n-1)p}.$$

\item $1<p<n$ and \begin{align*}
u(x)\leq Cr(x)^{-d},\quad\forall r(x)\geq C^{-1}
\end{align*}
for some positive numbers $C$ and $d$ with
\begin{align*}
d>\dfrac{(n-3p)(p-1)(n-p)}{p^2\left(n+2-3p\right)},
\end{align*}
where $r(x)=\mathrm{dist}(x,o)$ is the distance function from $x$ to some fixed point $o\in M^n$.
\end{enumerate}
Then $M^n$ is isometric to the Euclidean space $\mathbb{R}^n$ and the solution is an Aubin-Talenti bubble, i.e., $u$ is given by
\begin{align*}
u(x)=\left(a+b\abs{x-x_0}^{\frac{p}{p-1}}\right)^{-\frac{n-p}{p}},
\end{align*}
where $x_0\in\mathbb{R}^n$ is a fixed point, $a$ and $b$ are positive constants satisfying
$$ n\left(\dfrac{n-p}{p-1}\right)^{p-1}ab^{p-1}=1.$$
\end{main}

The study by Ou \cite{Ou22on} focuses on categorizing solutions to the critical $p$-Laplace equation in Euclidean space without any initial assumptions, thereby extending the results of Catino, Monticelli, and Roncoroni \cite{CatMonRon23on} to a wider range of $p$. Our research not only extends the applicability for various values of $p$ but also explores the rigidity of manifolds with nonnegative Ricci curvature that admit positive solutions to the critical $p$-Laplace equation.

In particular, we obtain the following corollary which was obtained by Catino, Monticelli and Roncoroni \cite[Theorem 1.2 (i) and Theorem 1.8]{CatMonRon23on}.

\begin{cor}
A complete, connected, noncompact and nonflat Riemannian surface with nonnegative Gaussian curvature does not admit any positive and weak solution to the critical $p$-Laplace equation \eqref{eq:critical-p} for each $1<p<2$.
\end{cor}

If $p=2$, we obtain the following corollary which was obtained by by Ciraolo, Farina and Polvara \cite[Theorem 1.1 (i)]{CirFarPol24classification} which extended a result of Catino, Monticelli and Roncoroni \cite[Theorem 1.2]{CatMonRon23on} and a result of Catino and Monticelli \cite[Corollary 1.5]{CatMon24semilinear}.
\begin{cor}
Let $M^n$ be a complete, connected and noncompact Riemannian manifold of dimension $n$ and with nonnegative Ricci curvature. If $n\in\set{3,4,5}$ and there exists a positive and weak solution to
\begin{align*}
-\Delta u=u^{\frac{n+2}{n-2}}
\end{align*}
in $M^n$, then $M^n$ is isometric to the Euclidean plane $\mathbb{R}^n$ and the solutions are given by the Aubin-Talenti bubbles.
\end{cor}

Observe that
$$n>\dfrac{(n^2-3n+1)p+n}{(n-1)p}.$$
It is established that any positive solution $u$ to the critical $p$-Laplace equation \eqref{eq:critical-p} possesses finite potential energy, i.e., $u\in L^{p_S}\left(M^n\right)$, which is equivalent to having finite kinetic energy, i.e., $\nabla u\in L^{p}\left(M^n\right)$, and further equivalent to possessing finite (total) energy, i.e., $u\in L^{p_S}\left(M^n\right)$ and $\nabla u\in L^{p}\left(M^n\right)$. An immediate consequence of the second part of \autoref{thm:main1}, which extend a result due to Catino and Monticelli \cite[Corollary 1.3]{CatMon24semilinear} for $p=2$, is the following which answers (C) of \autoref{prob:2}.

\begin{cor}
A complete, connected, noncompacted and nonflat $n$-dimensional Riemannian manifold $M^n$ with nonnegative Ricci curvature does not admit any Sobolev minimizer of $\mathcal{S}_p\left(M^n\right)$ for each $1<p<n$.
\end{cor}

In fact, F. Nobili and I.Y. Violo told us that they had proved this conclusion stated in the above corollary in \cite{Nob-Vio24} by a completely different method after we posted this paper on arXiv. It seems that their proof depends heavily on the optimal quantitative non-collapsing volume property and the sharp expression of Sobolev constant.

Additionally, we examine solutions that may possess infinite energy yet adhere to an appropriate condition at infinity. If $p=2$, then
\begin{align*}
\dfrac{(n-3p)(p-1)(n-p)}{p^2\left(n+2-3p\right)}=\dfrac{(n-6)(n-2)}{4(n-4)}.
\end{align*}
Consequently, we obtain a result of Ciraolo, Farina and Polvara \cite[Theorem 1.1, (ii)]{CirFarPol24classification} which extend a result of Catino and Monticelli \cite[Theorem 1.4]{CatMon24semilinear}.
\vspace{2ex}

Secondly, we examine the classification results for solutions to the Liouville equation \eqref{eq:critical-n} on an $n$-dimensional Riemannian manifold.
\vspace{2ex}

\begin{defn}A function $u\in W^{1,n}_{loc}\left(M^n\right)\cap L^{\infty}_{loc}\left(M^n\right)$ is said to be a weak solution to \eqref{eq:critical-n} if
\begin{align*}
\int \abs{\nabla u}^{p-2}\hin{\nabla u}{\nabla\phi}=\int e^{nu}\phi,\quad\forall \phi\in C_0^{\infty}\left(M^n\right).
\end{align*}
\end{defn}

The second main result, which extends a result by Ciraolo, Farina, and Polvara \cite[Theorem 1.2]{CirFarPol24classification}, is as follows:

\begin{main}\label{thm:main2}
Let $M^n$ be a complete Riemannian manifold of dimension $n$ and with nonnegative Ricci curvature. Let $F$ be any positive nondecreasing function satisfying
\begin{align*}
\int^{\infty}\dfrac{\dif s}{sF(s)}=\infty.
\end{align*}
If there is a weak solution $u$ to the Liouville equation \eqref{eq:critical-n} in $M^n$, i.e.,
\begin{align*}
-\Delta_nu=e^{nu},
\end{align*}
such that
\begin{align*}
u(x)\geq-\dfrac{n}{n-1}\ln\left(r(x)F^{1/2}\left(r(x)\right)\right)-C,\quad \forall r(x)\geq C^{-1},
\end{align*}
for some positive constant $C$, where $r(x)=\mathrm{dist}(x,o)$ is the distance function from $x$ to some fixed point $o\in M^n$, then $M^n$ is isometric to the Euclidean space $\mathbb{R}^n$ and the solution is given by
\begin{align*}
u(x)=-\ln\left(a+b\abs{x-x_0}^{\frac{n}{n-1}}\right) \quad\mbox{with}\quad \frac{n^n}{(n-1)^{n-1}}ab^{n-1}=1,
\end{align*}
where $x_0\in\mathbb{R}^n$ is a fixed point, $a>0$ and $ b>0$.
\end{main}

\subsection{Strategy of proof}

The strategy of our proof is to adopt the main idea of \cite{GidSpr81global, SerZou02cauchy, HeSunWan24optimal, Ou22on, CatMon24semilinear, CirFarPol24classification, Vet24note}, which is to state some integral estimates via the method of integration by parts. In order to state several new integral inequality and deal with the ``error" terms by more careful analysis, first of all, we need to exploit a new and sharp nonlinear Kato inequality (see \autoref{lem:preliminary-1} in \autoref{sec:preliminaries}). The basic idea and routine are as follows.

By setting
\begin{align*}
w=\left(\dfrac{n-p}{p}\right)^{\frac{p-1}{p}}u^{-\frac{p}{n-p}}\quad\mbox{and}\quad f=\dfrac{n(p-1)}{p}w^{-1}\abs{\nabla w}^{p}+w^{-1},
\end{align*}
if $u$ is a positive solution to the critical $p$-Laplace equation \eqref{eq:critical-p}, and setting
\begin{align*}
w=e^{-u}\quad\mbox{and}\quad f=(n-1)w^{-1}\abs{\nabla w}^{n}+w^{-1},
\end{align*}
if $u$ is a solution to the Liouville equation \eqref{eq:critical-n}, we find that
\begin{align*}
\Delta_pw=f.
\end{align*}
The function $f$ is referred to as the $P$-function, see \cite{CirFarPol24classification} for semilinear case. Denote by
\begin{align*}
\mathbf{W}=\abs{\nabla w}^{p-2}\nabla w \quad\mbox{and}\quad \mathbf{E}=\nabla\mathbf{W}-\dfrac{\Div\mathbf{W}}{n}g.
\end{align*}
Direct computation yields (see \autoref{lem:bochner})
\begin{align*}
\Div\left(w^{1-n}\mathbf{E}\left(\mathbf{W}\right)\right)=w^{1-n}\trace\mathbf{E}^2+w^{1-n}Ric\left(\mathbf{W},\mathbf{W}\right).
\end{align*}
We can express the vector field $\mathbf{E\left(W\right)}$ as follows (see \autoref{lem:EW})
\begin{align*}
\mathbf{E}\left(\mathbf{W}\right)=\dfrac{1}{n(p-1)}w\abs{\nabla w}^{p-2}\mathbf{A}\left(\nabla f\right)
\end{align*}
where
\begin{align*}
\mathbf{A}(X)=X+(p-2)\abs{\nabla w}^{-2}\hin{X}{\nabla w}\nabla w,\quad\forall X\in TM.
\end{align*}
We have
\begin{align*}
\dfrac{1}{n(p-1)}\Div\left(w^{2-n}\abs{\nabla w}^{p-2}\mathbf{A}\left(\nabla f\right)\right)=w^{1-n}\trace\mathbf{E}^2+w^{1-n}Ric\left(\mathbf{W},\mathbf{W}\right).
\end{align*}
The nonlinear Kato inequality (see the third part in \autoref{lem:preliminary-1}) gives
\begin{align*}
\trace\mathbf{E}^2\geq\dfrac{1}{(n-1)n(p-1)}w^2\abs{\nabla w}^{-2}\hin{\mathbf{A}\left(\nabla f\right)}{\nabla f}.
\end{align*}
Thus, if the Ricci curvature is nonnegative, then it follows from the above
\begin{align*}
\Div\left(w^{2-n}\abs{\nabla w}^{p-2}\mathbf{A}\left(\nabla f\right)\right)\geq\dfrac{1}{n-1}w^{3-n}\abs{\nabla w}^{-2}\hin{\mathbf{A}\left(\nabla f\right)}{\nabla f}.
\end{align*}

For each constant $\alpha$, there holds
\begin{align*}
\Div\left(f^{-\alpha}w^{2-n}\abs{\nabla w}^{p-2}\mathbf{A}\left(\nabla f\right)\right)=&f^{-\alpha}\Div\left(w^{2-n}\abs{\nabla w}^{p-2}\mathbf{A}\left(\nabla f\right)\right)-\alpha f^{-\alpha-1}w^{2-n}\abs{\nabla w}^{p-2}\hin{\mathbf{A}\left(\nabla f\right)}{\nabla f}.
\end{align*}
We conclude that
\begin{align*}
\Div\left(f^{-\alpha}w^{2-n}\abs{\nabla w}^{p-2}\mathbf{A}\left(\nabla f\right)\right)\geq\left(\dfrac{1}{n-1}w^{-1}+\left(\dfrac{n(p-1)}{(n-1)p}-\alpha \right)w^{-1}\abs{\nabla w}^p\right)f^{-\alpha-1}w^{3-n}\abs{\nabla w}^{-2}\hin{\mathbf{A}\left(\nabla f\right)}{\nabla f}.
\end{align*}
Hence, for every $\alpha$ satisfying
$$\alpha<\dfrac{n(p-1)}{(n-1)p}$$
there holds
\begin{align*}
\Div\left(f^{-\alpha}w^{2-n}\abs{\nabla w}^{p-2}\mathbf{A}\left(\nabla f\right)\right)\geq C_{\alpha}^{-1}f^{-\alpha}w^{3-n}\abs{\nabla w}^{-2}\hin{\mathbf{A}\left(\nabla f\right)}{\nabla f}.
\end{align*}
Here $C_{\alpha}$ is a positive constant depending only on $n$, $p$ and $\alpha$. Based on this inequality, one can prove that if $f$ is not a constant function, then for each real number
$$\alpha<\dfrac{n(p-1)}{(n-1)p}$$
we have (see \autoref{thm:basic-p} and \autoref{thm:basic-n})
\begin{align*}
\liminf_{R\to\infty}R^{-2}\int_{B_{R}\setminus B_{R/2}}f^{-\alpha}w^{1-n}\abs{\nabla w}^{2p-2}=\infty.
\end{align*}
Moreover, for every positive nondecreasing function $F(s)$ satisfying
\begin{align*}
\int^{\infty}\dfrac{\dif s}{sF(s)}=\infty,
\end{align*}
we have
\begin{align*}
\limsup_{R\to\infty}\dfrac{1}{R^2F(R)}\int_{B_{R}\setminus B_{R/2}}f^{-\alpha}w^{1-n}\abs{\nabla w}^{2p-2}=\infty.
\end{align*}
The rest of the proof is to prove the following integral estimate
\begin{align*}
\int_{B_{R}}f^{-\alpha}w^{1-n}\abs{\nabla w}^{2p-2}\leq C_{\alpha}R^{2}F(R)
\end{align*}
under our assumptions.

In order to prove the third part $(C)$ of \autoref{thm:main1}, we need also to establish Cheng-Yau type gradient estimate for positive solutions to the critical $p$-Laplace ($1<p<n$) equation by Moser iteration (see \autoref{thm:Cheng-Yau} in \autoref{sec:gradient-estimate}). Such gradient estimate for the case of $p=2$, observed first by Fogagnolo, Malchiodi and Mazzieri \cite{FogMalMaz23note}, can be established for solutions to the critical Laplace equation by the technique pioneered by Yau \cite{Yau75Uni} in order to obtain gradient estimates for harmonic functions under lower Ricci curvature bounds.
\medskip

The paper is organized as follows: In \autoref{sec:preliminaries}, we gather several useful inequalities and formulas from both algebraic and geometric contexts. Particularly, we introduce a significant nonlinear Kato inequality \eqref{eq:Kato}, which is essential in our proofs. In \autoref{sec:integral}, we establish several integral inequalities for solutions to either the critical $p$-Laplace equation or the Liouville equation. Although some of these inequalities are well-established, others present new developments. Specifically, the integral inequalities presented in \autoref{thm:basic-p} and \autoref{thm:basic-n} are new and play a crucial role in the application to classification results. The main idea behind the proofs of these integral inequalities originates from Karp \cite{Kar82subharmonic}. In \autoref{sec:gradient-estimate}, we employ Moser iteration to derive a local Cheng-Yau type gradient estimate for the critical $p$-Laplace equation when $1<p < n$. This gradient estimate is vital in studying the classification results for solutions with polynomial decay. In \autoref{sec:rigidity}, we discuss the rigidity results for solutions to the critical $p$-Laplace equations and Liouville equations, providing necessary proofs for related results, which constitutes the primary focus of this paper.

\section{Preliminaries}\label{sec:preliminaries}

Let $p > 1$ and $\Omega \subset M^n$ be a domain. Let $w \in C^3\left(\Omega\right)$ be a positive function satisfying the following regularity condition:
\begin{align}\label{eq:regularity}
\nabla w \neq 0, \quad \text{in } \Omega.
\end{align}
The vector field $\mathbf{W}$ defined in $\Omega$ is given by
\begin{align*}
\mathbf{W} = \abs{\nabla w}^{p-2} \nabla w.
\end{align*}
This vector field plays a crucial role in the study of the $p$-Laplacian. We consider the endomorphisms $\mathbf{E},\, \mathbf{A}: T\Omega \to T\Omega$, which are defined for every tangent vector $X \in T\Omega$ as follows:
\begin{align*}
\mathbf{E}(X) = \nabla_X \mathbf{W} - \dfrac{\Div \mathbf{W}}{n} X,
\end{align*}
and
\begin{align*}
\mathbf{A}(X) = X + (p-2)\abs{\nabla w}^{-2} \hin{X}{\nabla w}\nabla w,
\end{align*}
respectively. Denote by $\mathbf{E}^*$ the dual endomorphism of $\mathbf{E}$.

We collect several useful properties of $\mathbf{W}$, $\mathbf{E}$, and $\mathbf{A}$ as follows.

\begin{lem}\label{lem:preliminary-1}
\begin{enumerate}[(1)]
\item The endomorphism $\mathbf{EA}$ is self-dual, i.e.,
\begin{align}\label{eq:EA}
\mathbf{E}\mathbf{A} = \mathbf{A}\mathbf{E}^*.
\end{align}

\item The trace of $\mathbf{E}^2$ vanishes if and only if $\mathbf{E}$ vanishes when $p>1$. More specifically,
\begin{align}\label{eq:E}
\dfrac{2(p-1)}{1+(p-1)^2}\abs{\mathbf{E}}^2 \leq \trace\mathbf{E}^2 \leq \abs{\mathbf{E}}^2.
\end{align}

\item If $p>1$, then the following nonlinear Kato inequality holds:
\begin{align}\label{eq:Kato}
\trace\mathbf{E}^2 \geq \dfrac{n(p-1)}{n-1} \dfrac{\hin{\mathbf{E}(\mathbf{W})}{\mathbf{A}^{-1}(\mathbf{E}(\mathbf{W}))}}{\abs{\mathbf{W}}^2}.
\end{align}

\item If $p>1$, then
\begin{align}\label{eq:E'}
\abs{\mathbf{W}}^{-2}\abs{\mathbf{E}(\mathbf{W})}^2 \leq \max\set{\dfrac{n-1}{n},\, \dfrac{1}{2(p-1)}}\trace\mathbf{E}^2.
\end{align}
\end{enumerate}
\end{lem}

\begin{rem}
Here we would like to give some comments on the above lemma.
\begin{itemize}
\item It should be mentioned that the second inequality in Equation \eqref{eq:E} has been established in \cite{SerZou02cauchy}. We also emphasize that the first inequality in Equation \eqref{eq:E} was rigorously proven by He and the authors \cite[Lemma 3.3]{HeSunWan24optimal}, which plays a crucial role in the study of Liouville theorems for subcritical quasilinear equations on Riemannian manifolds. However, to address Liouville theorems in Euclidean space, it is sufficient to employ the weaker inequality $\trace\mathbf{E}^2 \geq 0$, as demonstrated in \cite{SerZou02cauchy}.

\item The inequality \eqref{eq:Kato} is sharp. To see this, one can consider the function
$$w(x)=-\dfrac{p-1}{n-p}\abs{x}^{-\frac{n-p}{p-1}}$$
in the Euclidean space $\mathbb{R}^n$. Direct computation yields $\Delta_pw=0$ and for all $x\neq 0$
\begin{align*}
\mathbf{W}=\abs{x}^{-n}x,\quad\mathbf{E}=\abs{x}^{-n}\left(\mathbf{Id}-n\abs{x}^{-2}xx^{T}\right).
\end{align*}
Thus
 \begin{align*}
 \mathbf{E}^2=\abs{x}^{-2n}\left(\mathbf{Id}+\left(n^2-2n\right)\abs{x}^{-2}xx^{T}\right),\quad \mathbf{E\left(W\right)}=(1-n)\abs{x}^{-2n}x,\quad \mathbf{A}^{-1}\left(\mathbf{E\left(W\right)}\right)=\dfrac{1-n}{p-1}\abs{x}^{-2n}x.
 \end{align*}
 We have
 \begin{align*}
 \trace\mathbf{E}^2=\left(n^2-n\right)\abs{x}^{-2n}\quad\quad\mbox{and}\quad\quad \dfrac{\hin{\mathbf{E}(\mathbf{W})}{\mathbf{A}^{-1}(\mathbf{E}(\mathbf{W}))}}{\abs{\mathbf{W}}^2}=\dfrac{\left(n-1\right)^2}{p-1}\abs{x}^{-2n},
 \end{align*}
 which imply
 \begin{align*}
 \trace\mathbf{E}^2=\dfrac{n(p-1)}{n-1}\dfrac{\hin{\mathbf{E}(\mathbf{W})}{\mathbf{A}^{-1}(\mathbf{E}(\mathbf{W}))}}{\abs{\mathbf{W}}^2},
 \quad\forall x\neq0.
 \end{align*}
 \item
 If $p=2$ and $w$ is a harmonic function, i.e., $\Delta w=0$, then the inequality \eqref{eq:Kato} simplifies to
 \begin{align*}
 \abs{\nabla^2w}^2\geq\dfrac{n}{n-1}\abs{\nabla\abs{\nabla w}}^2
 \end{align*}
 provided $\nabla w\neq0$.
 This is the well-known Kato inequality for harmonic functions.
\end{itemize}
\end{rem}

\begin{proof}[Proof of \autoref{lem:preliminary-1}]
{\bf Firstly, we prove that $\mathbf{EA}$ is a self-dual endomorphism, i.e., we prove the identity \eqref{eq:EA}.}

This is a straightforward computation. In fact, we compute for each tangent vector $X\in T\Omega$
\begin{align*}
\nabla_X\mathbf{W}=&\nabla_X\left(\abs{\nabla w}^{p-2}\nabla w\right)\\
=&\abs{\nabla w}^{p-2}\nabla_X\nabla w+(p-2)\abs{\nabla w}^{p-4}\hin{\nabla_X\nabla w}{\nabla w}\nabla w\\
=&\abs{\nabla w}^{p-2}\left(\nabla_X\nabla w+(p-2)\abs{\nabla w}^{-2}\hin{\nabla_X\nabla w}{\nabla w}\nabla w\right).
\end{align*}
That is
\begin{align}\label{eq:EA1}
\nabla_X\mathbf{W}=\abs{\nabla w}^{p-2}\mathbf{A}\left(\nabla_X\nabla w\right).
\end{align}
Since
\begin{align*}
\hin{\nabla_X\nabla w}{Y}=\hin{\nabla_Y\nabla w}{X},\quad\forall X, Y\in T\Omega,
\end{align*}
we conclude from \eqref{eq:EA1}
\begin{align*}
\hin{\nabla_{\mathbf{A}\left(X\right)}\mathbf{W}}{Y}=&\hin{\abs{\nabla w}^{p-2}\mathbf{A}\left(\nabla_{\mathbf{A}\left(X\right)}\nabla w\right)}{Y}\\
=&\hin{\abs{\nabla w}^{p-2}\nabla_{\mathbf{A}\left(X\right)}\nabla w}{\mathbf{A}\left(Y\right)}\\
=&\hin{\abs{\nabla w}^{p-2}\nabla_{\mathbf{A}\left(Y\right)}\nabla w}{\mathbf{A}\left(X\right)}\\
=&\hin{\abs{\nabla w}^{p-2}\mathbf{A}\left(\nabla_{\mathbf{A}\left(Y\right)}\nabla w\right)}{X},
\end{align*}
which implies
\begin{align}\label{eq:EA2}
\hin{\nabla_{\mathbf{A}\left(X\right)}\mathbf{W}}{Y}=\hin{\nabla_{\mathbf{A}\left(Y\right)}\mathbf{W}}{X}.
\end{align}
Here we have used the fact that the endomorphism $\mathbf{A}$ is self-dual. Consequently, it follows from the definition of the endomorphism $\mathbf{E}$ and \eqref{eq:EA2}
\begin{align*}
\hin{\mathbf{E}\left(\mathbf{A}\left(X\right)\right)}{Y}=&\hin{\nabla_{\mathbf{A}\left(X\right)}\mathbf{W}-\dfrac{\Div\mathbf{W}}{n}\mathbf{A}\left(X\right)}{Y}\\
=&\hin{\nabla_{\mathbf{A}\left(Y\right)}\mathbf{W}-\dfrac{\Div\mathbf{W}}{n}\mathbf{A}\left(Y\right)}{X}\\
=&\hin{\mathbf{E}\left(\mathbf{A}\left(Y\right)\right)}{X}\\
=&\hin{Y}{\mathbf{A}\left(\mathbf{E}^*\left(X\right)\right)}.
\end{align*}
Thus, we obtain \eqref{eq:EA}.
\vspace{2ex}

{\bf Secondly, we prove that $\trace\mathbf{E}^2=0$ if and only if $\abs{\mathbf{E}}^2=0$. More specifically, we will verify the inequality \eqref{eq:E}.}

Since $\abs{\mathbf{E}}^2=\trace\left(\mathbf{E}\mathbf{E}^*\right)$, it follows from the Cauchy-Schwarz inequality
\begin{align*}
\trace\mathbf{E}^2\leq\abs{\mathbf{E}}^2.
\end{align*}
To prove the first inequality of \eqref{eq:E}, by observing $\nabla w\neq0$, we can choose a local orthonormal frames $\set{e_i}$ at a considering point such that
\begin{align*}
 \nabla w=\abs{\nabla w}e_1.
\end{align*}
We have
\begin{align*}
\mathbf{A}e_i=\lambda_ie_i,
\end{align*}
where $\lambda_1=p-1, \lambda_2=\lambda_3=\dotsm=\lambda_n=1$. Denote by $\mathbf{E}_{ij}=\hin{\mathbf{E}e_i}{e_j}$. It follows from \eqref{eq:EA}
\begin{align*}
\mathbf{E}_{ij}\lambda_j=\lambda_i\mathbf E_{ji},\quad\forall i, j.
\end{align*}
Thus, for every $i, j\in\set{2,3,\dots,n}$
\begin{align*}
(p-1)\mathbf{E}_{i1}=\mathbf{E}_{1i},\quad \mathbf{E}_{ij}=\mathbf{E}_{ji}.
\end{align*}
Up to a rotation without loss of generality, we may assume
\begin{align*}
\mathbf{E}_{ij}=\mathbf{E}_{ji}=0,\quad 2\leq i<j\leq n.
\end{align*}
We have
\begin{align}\label{eq:E1}
\trace\mathbf{E}^2=\sum_{i,j=1}^n\mathbf{E}_{ij}\mathbf{E}_{ji}=\sum_{j=1}^n\mathbf{E}_{jj}^2+2(p-1)\sum_{j=2}^n\mathbf{E}_{1j}^2
\end{align}
and
\begin{align}\label{eq:E2}
 \abs{\mathbf{E}}^2=\sum_{i,j=1}^n\mathbf{E}_{ij}^2=\sum_{j=1}^n\mathbf{E}_{jj}^2+\left(1+(p-1)^2\right)\sum_{j=2}^n\mathbf{E}_{1j}^2.
\end{align}
We obtain from \eqref{eq:E1} and \eqref{eq:E2}
\begin{align*}
 \trace\mathbf{E}^2-\dfrac{2(p-1)}{1+(p-1)^2}\abs{\mathbf{E}}^2=\dfrac{(p-2)^2}{1+(p-1)^2}\sum_{j=1}^n\mathbf{E}_{jj}^2
\end{align*}
which implies
\begin{align*}
 \dfrac{2(p-1)}{1+(p-1)^2}\abs{\mathbf{E}}^2\leq\trace\mathbf{E}^2.
\end{align*}
We complete the proof of \eqref{eq:E}.

\vspace{2ex}

{\bf Thirdly, we prove the nonlinear Kato inequality. In other words, we will check the inequality \eqref{eq:Kato}.}

To demonstrate this, we employ the previously defined notation and establish the inequality in the manner outlined below. Since $\trace\mathbf{E}=0$, the Cauchy-Schwarz inequality gives
\begin{align}\label{eq:Kato1}
\sum_{j=1}^n\mathbf{E}_{jj}^2\geq&\mathbf{E}_{11}^2+\dfrac{1}{n-1}\left(\sum_{j=2}^n\mathbf{E}_{jj}\right)^2=\dfrac{n}{n-1}\mathbf{E}_{11}^2.
\end{align}
Direct computation yields for $p>1$
\begin{align}\label{eq:Kato2}
\hin{\mathbf{E}e_1}{\mathbf{A}^{-1}\mathbf{E}e_1}=&\dfrac{1}{p-1}\mathbf{E}_{11}^2+\sum_{j=2}^n\mathbf{E}_{1j}^2.
\end{align}
Thus, it follows from \eqref{eq:E1}, \eqref{eq:Kato1} and \eqref{eq:Kato2}
\begin{align*}
\trace\mathbf{E}^2\geq&\dfrac{n}{n-1}\mathbf{E}_{11}^2+2(p-1)\sum_{j=2}^n\mathbf{E}_{1j}^2\\
\geq&\dfrac{n(p-1)}{n-1}\hin{\mathbf{E}e_1}{\mathbf{A}^{-1}\mathbf{E}e_1}.
\end{align*}
We obtain the nonlinear Kato inequality \eqref{eq:Kato}.
\vspace{2ex}

{\bf Finally, one can check \eqref{eq:E'}.}

In fact, the above argument also yields
\begin{align*}
\abs{\mathbf{E}e_1}^2=&\sum_{j=1}^n\mathbf{E}_{1j}^2\\
\leq&\max\set{\dfrac{n-1}{n},\dfrac{1}{2(p-1)}}\trace\mathbf{E}^2.
\end{align*}
That is, we obtain the inequality \eqref{eq:E'}.
\end{proof}

The next lemma provides a useful expression for the vector field $\mathbf{E}\left(\mathbf{W}\right)$.
\begin{lem}\label{lem:EW}
Assume $w\in C^3\left(\Omega\right)$ is a positive solution to the following quasilinear equation:
\begin{align}\label{eq:critical-1}
 \Delta_pw=f\coloneqq\dfrac{n(p-1)}{p}w^{-1}\abs{\nabla w}^{p}+w^{-1},\quad\text{in}\ \Omega,
\end{align}
and satisfies the regularity condition \eqref{eq:regularity}.
We can express the vector field $\mathbf{E}\left(\mathbf{W}\right)$ as follows
\begin{align}\label{eq:X}
 \mathbf{E}\left(\mathbf{W}\right)=\dfrac{1}{n(p-1)}w\abs{\nabla w}^{p-2}\mathbf{A}\left(\nabla f\right).
\end{align}
\end{lem}
\begin{proof}
In fact, by definition, $\mathbf{W}=\abs{\nabla w}^{p-2}\nabla w$, we obtain from \eqref{eq:EA1}
\begin{align*}
\nabla_{\mathbf{W}}\mathbf{W}=\abs{\nabla w}^{p-2}\mathbf{A}\left(\nabla_{\mathbf{W}}\nabla w\right)=\abs{\nabla w}^{2p-4}\mathbf{A}\left(\nabla_{\nabla w}\nabla w\right)
\end{align*}
which gives
\begin{align}\label{eq:WW}
\nabla_{\mathbf{W}}\mathbf{W}=\dfrac{1}{p}\abs{\nabla w}^{p-2}\mathbf{A}\left(\nabla\abs{\nabla w}^p\right).
\end{align}
Thus, by definition,
$$\mathbf{E}\left(\mathbf{W}\right)=\nabla_{\mathbf{W}}\mathbf{W}-\dfrac{\Div\mathbf{W}}{n}\mathbf{W},$$
we obtain
\begin{align*}
\mathbf{E}\left(\mathbf{W}\right)=&\dfrac{1}{p}\abs{\nabla w}^{p-2}\mathbf{A}\left(\nabla\abs{\nabla w}^p\right)-\dfrac1n\left(\dfrac{n(p-1)}{p}w^{-1}\abs{\nabla w}^{p}+w^{-1}\right)\mathbf{W}\\
 =&\dfrac{1}{p}\abs{\nabla w}^{p-2}\mathbf{A}\left(\nabla\abs{\nabla w}^p\right)-\dfrac1n\left(\dfrac{n}{p}\abs{\nabla w}^{p}+\dfrac{1}{p-1}\right)w^{-1}\abs{\nabla w}^{p-2}\mathbf{A}\left(\nabla w\right)\\
 =&w\abs{\nabla w}^{p-2}\mathbf{A}\left(\nabla\left(\dfrac1pw^{-1}\abs{\nabla w}^p+\dfrac{1}{n(p-1)}w^{-1}\right)\right).
\end{align*}
Here we have used the fact that $w$ satisfies the partial differential equation \eqref{eq:critical-1}. Consequently, we obtain \eqref{eq:X}.

\end{proof}

In the end of this section, we state a useful Bochner formula for $p$-Laplacian.
\begin{lem}\label{lem:bochner}
 If $w\in C^3\left(\Omega\right)$ solves the quasilinear equation \eqref{eq:critical-1} with the regularity condition \eqref{eq:regularity}, then we have
\begin{align}\label{eq:bochner}
 \Div\left(w^{1-n}\mathbf{E}\left(\mathbf{W}\right)\right)=w^{1-n}\trace\mathbf{E}^2+w^{1-n}Ric\left(\mathbf{W},\mathbf{W}\right).
\end{align}
\end{lem}
\begin{proof}
We begin with the following Bochner formula
\begin{align}\label{eq:bochner-1}
 \Div\nabla_{\mathbf{W}}\mathbf{W}=\sum_{i=1}^n\hin{\nabla_{\nabla_{e_i}\mathbf{W}}\mathbf{W}}{e_i}+\hin{\nabla\Div\mathbf{W}}{\mathbf{W}}+Ric\left(\mathbf{W},\mathbf{W}\right),
\end{align}
where $\set{e_i}$ is a local orthonormal tangent frames.
 This is a straightforward verification. In fact, we may assume $\nabla e_i=0$ at a considering point and compute
\begin{align*}
 \Div\nabla_{\mathbf{W}}\mathbf{W}=&\sum_{i=1}^n\hin{\nabla_{e_i}\nabla_{\mathbf{W}}\mathbf{W}}{e_i}\\
 =&\sum_{i=1}^n\hin{R\left(e_i,\mathbf{W}\right)\mathbf{W}+\nabla_{\mathbf{W}}\nabla_{e_i}\mathbf{W}+\nabla_{[e_i,\mathbf{W}]}\mathbf{W}}{e_i}\\
 =&Ric\left(\mathbf{W},\mathbf{W}\right)+\sum_{i=1}^n\mathbf{W}\hin{\nabla_{e_i}\mathbf{W}}{e_i}+\sum_{i=1}^n\hin{\nabla_{\nabla_{e_i}\mathbf{W}}\mathbf{W}}{e_i}\\
 =&\sum_{i=1}^n\hin{\nabla_{\nabla_{e_i}\mathbf{W}}\mathbf{W}}{e_i}+\hin{\nabla\Div\mathbf{W}}{\mathbf{W}}+Ric\left(\mathbf{W},\mathbf{W}\right).
\end{align*}
We obtain the desired Bochner formula \eqref{eq:bochner-1}.

Notice that
\begin{align*}
 \trace\mathbf{E}^2=&\sum_{i=1}^n\hin{\mathbf{E}\left(\mathbf{E}\left(e_i\right)\right)}{e_i}\\
 =&\sum_{i,j=1}^n\hin{\mathbf{E}\left(e_i\right)}{e_j}\hin{\mathbf{E}\left(e_j\right)}{e_i}\\
 =&\sum_{i,j=1}^n\left(\hin{\nabla_{e_i}\mathbf{W}}{e_j}-\dfrac{\Div\mathbf{W}}{n}\delta_{ij}\right)\left(\hin{\nabla_{e_j}\mathbf{W}}{e_i}-\dfrac{\Div\mathbf{W}}{n}\delta_{ji}\right)\\
 =&\sum_{i,j=1}^n\hin{\nabla_{e_i}\mathbf{W}}{e_j}\hin{\nabla_{e_j}\mathbf{W}}{e_i}-\dfrac1n\abs{\Div\mathbf{W}}^2.
\end{align*}
In other words,
\begin{align}\label{eq:E^2}
 \trace\mathbf{E}^2=\sum_{i=1}^n\hin{\nabla_{\nabla_{e_i}\mathbf{W}}\mathbf{W}}{e_i}-\dfrac1n\abs{\Div\mathbf{W}}^2.
\end{align}
By the definition of $\mathbf{E}\left(\mathbf{W}\right)$, we obtain from \eqref{eq:critical-1}, \eqref{eq:X}, \eqref{eq:bochner-1} and \eqref{eq:E^2}
\begin{align*}
 \Div\mathbf{E}\left(\mathbf{W}\right)=&\Div\left(\nabla_{\mathbf{W}}\mathbf{W}-\dfrac{\Div\mathbf{W}}{n}\mathbf{W}\right)\\
 =&\left(\sum_{i=1}^n\hin{\nabla_{\nabla_{e_i}\mathbf{W}}\mathbf{W}}{e_i}+\hin{\nabla\Div\mathbf{W}}{\mathbf{W}}+Ric\left(\mathbf{W},\mathbf{W}\right)\right)-\dfrac1n\left(\hin{\nabla\Div\mathbf{W}}{\mathbf{W}}+\abs{\Div\mathbf{W}}^2\right)\\
 =&\trace\mathbf{E}^2+Ric\left(\mathbf{W},\mathbf{W}\right)+\dfrac{n-1}{n}\hin{\nabla f}{\mathbf{W}}\\
 =&\trace\mathbf{E}^2+Ric\left(\mathbf{W},\mathbf{W}\right)+(n-1)w^{-1}\hin{\mathbf{E}\left(\mathbf{W}\right)}{\nabla w}.
\end{align*}
Thus
\begin{align*}
 \Div\left(w^{1-n}\mathbf{E}\left(\mathbf{W}\right)\right)=&w^{1-n}\Div\mathbf{E}\left(\mathbf{W}\right)+(1-n)w^{-n}\hin{\mathbf{E}\left(\mathbf{W}\right)}{\nabla w}\\
 =&w^{1-n}\left(\trace\mathbf{E}^2+Ric\left(\mathbf{W},\mathbf{W}\right)\right).
\end{align*}
We obtain the desired formula \eqref{eq:bochner}.

\end{proof}

\section{Integral inequalities}\label{sec:integral}

We begin with the following preliminary integral inequality which can be seen as a generalization of Karp's integral inequality for positive harmonic functions \cite{Kar82subharmonic}. Karp used his integral inequality to give some useful applications in geometric analysis.

We consider first the critical $p$-Laplace equation. Here and in the sequel, by $B_{R}=B_R(o)$ we shall mean a ball of radius $R$ and center $o\in M^n$.

\begin{theorem}\label{thm:basic-p}
Assume $1<p<n$ and $u\in W_{loc}^{1,p}\left(M^n\right)$ is a positive and weak solution to the critical $p$-Laplace equation \eqref{eq:critical-p} on a complete Riemannian manifold $M^n$ with dimension $n$ and nonnegative Ricci curvature. Let $F$ be any positive nondecreasing function satisfying
\begin{align*}
\int^{\infty}\dfrac{\dif s}{sF(s)}=\infty.
\end{align*}
 Denote by
\begin{align*}
w=\left(\dfrac{n-p}{p}\right)^{\frac{p-1}{p}}u^{-\frac{p}{n-p}}\quad\quad\mbox{and}\quad\quad f=\dfrac{n(p-1)}{p}w^{-1}\abs{\nabla w}^{p}+w^{-1}.
\end{align*}
If $f$ is not a constant function, then we have for each real number $\alpha<\frac{n(p-1)}{(n-1)p}$
\begin{align}\label{eq:K1}
\liminf_{R\to\infty}\dfrac{1}{R^2}\int_{B_{R}\setminus B_{R/2}}f^{-\alpha}w^{1-n}\abs{\nabla w}^{2p-2}=\infty,
\end{align}
and
\begin{align}\label{eq:K2}
\limsup_{R\to\infty}\dfrac{1}{R^2F(R)}\int_{B_{R}\setminus B_{R/2}}f^{-\alpha}w^{1-n}\abs{\nabla w}^{2p-2}=\infty.
\end{align}

\end{theorem}

\begin{proof}

 It is well known that for some constant $\hat\alpha\in(0,1)$ (cf. \cite{Dib83local, Tol84regularity,Uhl77regularity,AntCirFar23interior})
 \begin{align*}
 u\in C^{\infty}_{loc}\left(M^n\setminus Z\right)\cap C^{1,\hat\alpha}_{loc}\left(M^n\right)
 \end{align*}
 and
 \begin{align*}
 \abs{\nabla u}^{p-2}\nabla u\in L^{2}_{loc}\left(M^n\right),\quad \abs{\nabla u}^{p-2}\nabla^2 u\in L^{2}_{loc}\left(M^n\right).
 \end{align*}
Here $Z=\set{\nabla u=0}$ is the critical set of the solution $u$ which has zero measure \cite[Proposition 1.6]{AntCirFar23interior}. Moreover, if $1<p\leq2$,
\begin{align*}
 u\in W_{loc}^{2,2}\left(M^n\right).
\end{align*}
Denote by
\begin{align*}
\mathbf{W}=\abs{\nabla w}^{p-2}\nabla w \quad\quad\mbox{and}\quad\quad \mathbf{E}=\nabla\mathbf{W}-\dfrac{\Div\mathbf{W}}{n}g.
\end{align*}
It follows from \eqref{eq:X}
\begin{align*}
w^{1-n}\mathbf{E}\left(\mathbf{W}\right)=\dfrac{1}{n(p-1)}w^{2-n}\abs{\nabla w}^{p-2}\mathbf{A}\left(\nabla f\right)\in L^1_{loc}\left(M^n\right).
\end{align*}

Direct computation yields that
\begin{align*}
\Delta_pw=f
\end{align*}
holds true pointwisely in $M^n\setminus Z$. In fact,
\begin{align*}
\Delta_pu^{-\frac{p}{n-p}}=&-\left(\dfrac{p}{n-p}\right)^{p-1}\Div\left(u^{-\frac{n(p-1)}{n-p}}\abs{\nabla u}^{p-2}\nabla u\right)\\
=&-\left(\dfrac{p}{n-p}\right)^{p-1}\left(-\dfrac{n(p-1)}{n-p}u^{-\frac{(n-1)p}{n-p}}\abs{\nabla u}^p+u^{-\frac{n(p-1)}{n-p}}\Delta_pu\right)\\
=&\dfrac{n(p-1)}{p}u^{\frac{p}{n-p}}\abs{\nabla u^{-\frac{p}{n-p}}}^{p}+\left(\dfrac{p}{n-p}\right)^{p-1}u^{\frac{p}{n-p}}
\end{align*}
which yields
\begin{align*}
\Delta_pw=\dfrac{n(p-1)}{p}w^{-1}\abs{\nabla w}^{p}+w^{-1}.
\end{align*}

Since the Ricci curvature of $M^n$ is nonnegative, it follows from \eqref{eq:bochner} that
\begin{align*}
\Div\left(w^{1-n}\mathbf{E}\left(\mathbf{W}\right)\right)\geq w^{1-n}\trace\mathbf{E}^2
\end{align*}
holds in the sense of distribution, i.e., for every nonnegative function $\phi\in C_0^{\infty}\left(M^n\right)$
\begin{align*}
\int w^{1-n}\trace\mathbf{E}^2\phi\leq-\int\hin{w^{1-n}\mathbf{E}\left(\mathbf{W}\right)}{\nabla\phi}.
\end{align*}
For each real number $\gamma\geq2$ and nonnegative function $\eta\in C_0^{\infty}\left(M^n\right)$, the function $f^{-\alpha}\eta^{\gamma}$ is a well defined test function. Particularly, we have
\begin{align}\label{eq:Karp1}
 \int f^{-\alpha}w^{1-n}\trace\mathbf{E}^2\eta^{\gamma}\leq \alpha\int f^{-\alpha-1}w^{1-n}\hin{\mathbf{E}\left(\mathbf{W}\right)}{\nabla f}\eta^{\gamma}-\gamma\int f^{-\alpha}w^{1-n}\hin{\mathbf{E}\left(\mathbf{W}\right)}{\nabla \eta}\eta^{\gamma-1}.
\end{align}

It follows from \eqref{eq:Kato} and \eqref{eq:X}
\begin{align}\label{eq:Karp2}
\hin{\mathbf{E}\left(\mathbf{W}\right)}{\nabla f}=&n(p-1)w^{-1}\abs{\nabla w}^{2-p}\hin{\mathbf{E}\left(\mathbf{W}\right)}{\mathbf{A}^{-1}\left(\mathbf{E}\left(\mathbf{W}\right)\right)}\notag\\
\leq&(n-1)w^{-1}\abs{\nabla w}^{p}\trace\mathbf{E}^2.
\end{align}
We also have from \eqref{eq:E'}
\begin{align}\label{eq:Karp3}
\abs{\mathbf{E}\left(\mathbf{W}\right)}\leq&\sqrt{\max\set{\dfrac{n-1}{n},\dfrac{1}{2(p-1)}}}\abs{\nabla w}^{p-1}\sqrt{\trace\mathbf{E}^2}.
\end{align}
Inserting \eqref{eq:Karp1} and \eqref{eq:Karp2} into \eqref{eq:Karp3}, we get
\begin{align}\label{eq:Karp11'}
\begin{split}
 \int f^{-\alpha}w^{1-n}\trace\mathbf{E}^2\eta^{\gamma}\leq& (n-1)\alpha^+\int f^{-\alpha-1}w^{-n}\abs{\nabla w}^p\trace\mathbf{E}^2\eta^{\gamma}+C\gamma\int f^{-\alpha}w^{1-n}\abs{\nabla w}^{p-1}\sqrt{\trace\mathbf{E}^2}\abs{\nabla\eta}\eta^{\gamma-1}\\
\leq&\dfrac{(n-1)p}{n(p-1)}\alpha^+\int f^{-\alpha}w^{1-n}\trace\mathbf{E}^2\eta^{\gamma}+C\gamma\int f^{-\alpha}w^{1-n}\abs{\nabla w}^{p-1}\sqrt{\trace\mathbf{E}^2}\abs{\nabla\eta}\eta^{\gamma-1}.
\end{split}
\end{align}
Here we used the fact
\begin{align*}
 f\geq \dfrac{n(p-1)}{p}w^{-1}\abs{\nabla w}^p.
\end{align*}
Since
$$\alpha<\dfrac{n(p-1)}{(n-1)p},$$
the above inequality \eqref{eq:Karp11'} implies
\begin{align*}
\int f^{-\alpha}w^{1-n}\trace\mathbf{E}^2\eta^{\gamma}\leq C_{\alpha}\gamma\int f^{-\alpha}w^{1-n}\abs{\nabla w}^{p-1}\sqrt{\trace\mathbf{E}^2}\abs{\nabla\eta}\eta^{\gamma-1}.
\end{align*}
Applying H\"older's inequality with the respective exponent pair $\left(2,2\right)$, we have
\begin{align}\label{eq:Karp4}
 \int f^{-\alpha}w^{1-n}\trace\mathbf{E}^2\eta^{\gamma}\leq C_{\alpha}\gamma\left(\int_{\set{\nabla\eta\neq0}} f^{-\alpha}w^{1-n}\trace\mathbf{E}^2\eta^{\gamma}\right)^{1/2}\left(\int f^{-\alpha}w^{1-n}\abs{\nabla w}^{2p-2}\abs{\nabla\eta}^2\eta^{\gamma-2}\right)^{1/2}.
\end{align}

For each positive number $R$, we choose $\eta=\eta_{R}\in C_0^{\infty}\left(B_{R}\right)$ satisfying
\begin{align*}
 \eta_{R}\vert_{B_{R/2}}=1,\quad 0\leq\eta_{R}\leq 1,\quad \abs{\nabla\eta_{R}}\leq\frac{4}{R}
\end{align*}
and conclude from \eqref{eq:Karp4}
\begin{align*}
 \int_{B_{R}} f^{-\alpha}w^{1-n}\trace\mathbf{E}^2\eta_{R}^{2}\leq \dfrac{C_{\alpha}}{R}\left(\int_{B_{R}\setminus B_{R/2}} f^{-\alpha}w^{1-n}\trace\mathbf{E}^2\eta_{R}^{2}\right)^{1/2}\left(\int_{B_{R}\setminus B_{R/2}} f^{-\alpha}w^{1-n}\abs{\nabla w}^{2p-2}\right)^{1/2}
\end{align*}
which implies that for every sequence of positive numbers $\set{R_j}$ with $R_{j+1}\geq 2R_j$,
\begin{align}\label{eq:Karp5}
 G_{j+1}^2\leq C_{\alpha}\left(G_{j+1}-G_j\right)H_{j+1}
\end{align}
where
\begin{align*}
 G_{j}=\int_{B_{R_{j}}}f^{-\alpha}w^{1-n}\trace\mathbf{E}^2\eta_{R_{j}}^{2},\quad H_{j}=R_{j}^{-2}\int_{B_{R_{j}}\setminus B_{R_{j}/2}}f^{-\alpha}w^{1-n}\abs{\nabla w}^{2p-2}.
\end{align*}
In fact, it suffices to observe
\begin{align*}
 \int_{B_{R_{j+1}}\setminus B_{R_{j+1}/2}}f^{-\alpha}w^{1-n}\trace\mathbf{E}^2\eta_{R_{j+1}}^{2}=&\int_{B_{R_{j+1}}}f^{-\alpha}w^{1-n}\trace\mathbf{E}^2\eta_{R_{j+1}}^{2}-\int_{B_{R_{j+1}/2}}f^{-\alpha}w^{1-n}\trace\mathbf{E}^2\\
 \leq&\int_{B_{R_{j+1}}}f^{-\alpha}w^{1-n}\trace\mathbf{E}^2\eta_{R_{j+1}}^{2}-\int_{B_{R_{j}}}f^{-\alpha}w^{1-n}\trace\mathbf{E}^2\\
 \leq&\int_{B_{R_{j+1}}}f^{-\alpha}w^{1-n}\trace\mathbf{E}^2\eta_{R_{j+1}}^{2}-\int_{B_{R_{j}}}f^{-\alpha}w^{1-n}\trace\mathbf{E}^2\eta_j^2
\end{align*}
since $R_{j+1}\geq 2R_{j}$.

We can prove the first inequality \eqref{eq:K1} as follows. If
\begin{align*}
 \liminf_{R\to\infty}R^{-2}\int_{B_{R}\setminus B_{R/2}}f^{-\alpha}w^{1-n}\abs{\nabla w}^{2p-2}<\infty,
\end{align*}
then there exists a subsequence of positive numbers $R_{j}$ such that $R_{j+1}\geq 2R_{j}$ and
\begin{align*}
 \int_{B_{R_j}\setminus B_{R_j/2}}f^{-\alpha}w^{1-n}\abs{\nabla w}^{2p-2}\leq CR_{j}^2.
\end{align*}
Thus $H_{j}$ is uniformly bounded and we obtain from \eqref{eq:Karp5}
\begin{align*}
 G_{j+1}^2\leq C\left(G_{j+1}-G_{j}\right)
\end{align*}
which implies that $G_j$ is also uniformly bounded and
\begin{align*}
 \sum_{j=1}^{\infty}G_{j}^2<\infty.
\end{align*}
Consequently,
\begin{align*}
 \int_Mf^{-\alpha}w^{1-n}\trace\mathbf{E}^2=\lim_{j\to\infty}G_j=0.
\end{align*}
We deduce that $\trace\mathbf{E}^2\equiv0$, implying that $f$ must be a constant function.

 Now we will prove the second inequality \eqref{eq:K2}.
 For each positive integer number $j$, we set $R_{j}=2^j$. On the one hand, if $f$ is not a constant function, then $\trace\mathbf{E}^2\neq0$ and we may assume
 \begin{align*}
 G_{j}\geq C^{-1},\quad\forall j\geq1.
 \end{align*}
 On the other hand, if
 \begin{align*}
 \limsup\dfrac{1}{R^2F(R)}\int_{B_{R}\setminus B_{R/2}}f^{-\alpha}w^{1-n}\abs{\nabla w}^{2p-2}<\infty,
 \end{align*}
 then we have
 \begin{align*}
 H_j\leq CF\left(2^j\right),\quad\forall j\geq1.
 \end{align*}
 Thus, it follows from \eqref{eq:Karp5}
 \begin{align*}
 0<G_{j-1}G_{j}\leq G_{j}^2\leq C \left(G_j-G_{j-1}\right)H_j\leq C \left(G_j-G_{j-1}\right)F\left(2^j\right)
 \end{align*}
 which implies
 \begin{align*}
 \dfrac{1}{F\left(2^j\right)}\leq C\left(\dfrac{1}{G_{j-1}}-\dfrac{1}{G_j}\right).
 \end{align*}
 Thus
 \begin{align*}
 \sum_{j=1}^{\infty}\dfrac{1}{F\left(2^j\right)}\leq\dfrac{C}{G_0}.
 \end{align*}
 Consequently, it follows from the Cauchy integral test
 \begin{align*}
 \int^{\infty}\dfrac{\dif s}{sF(s)}<\infty
 \end{align*}
 which is a contradiction.
\end{proof}

For the Liouville equation, employing a reasoning analogous to that presented in \autoref{thm:basic-p}, we deduce that

\begin{theorem}\label{thm:basic-n}
Assume $u\in W_{loc}^{1,n}\left(M^n\right)$ is a weak solution to the Liouville equation \eqref{eq:critical-n} on a complete Riemannian manifold $M^n$ with dimension $n$ and nonnegative Ricci curvature. Denote by
\begin{align*}
 w=e^{-u} \quad\quad\mbox{and}\quad\quad f=(n-1)w^{-1}\abs{\nabla w}^{n}+w^{-1}.
\end{align*}
Let $F$ be any positive nondecreasing function satisfying
\begin{align*}
 \int^{\infty}\dfrac{\dif s}{sF(s)}=\infty.
\end{align*}
If $f$ is not a constant function, then we have for each real number $\alpha<1$
\begin{align*}
\liminf_{R\to\infty}\dfrac{1}{R^2}\int_{B_{R}\setminus B_{R/2}}f^{-\alpha}w^{1-n}\abs{\nabla w}^{2n-2}=\infty,
\end{align*}
and
\begin{align*}
\limsup_{R\to\infty}\dfrac{1}{R^2F(R)}\int_{B_{R}\setminus B_{R/2}}f^{-\alpha}w^{1-n}\abs{\nabla w}^{2n-2}=\infty.
\end{align*}
\end{theorem}

\begin{proof}
We compute
\begin{align*}
\Delta_nw=&-\Div\left(e^{-(n-1)u}\abs{\nabla u}^{n-2}\nabla u\right)\\
=&(n-1)e^{-(n-1)u}\abs{\nabla u}^n-e^{-(n-1)u}\Delta_nu\\
=&(n-1)w^{-1}\abs{\nabla w}^n+w^{-1}.
\end{align*}
Thus
\begin{align*}
\Delta_nw=(n-1)w^{-1}\abs{\nabla w}^n+w^{-1}.
\end{align*}

The rest of the proof is similar to \autoref{thm:basic-p}.
\end{proof}

To address the higher-dimensional case, the following is required:

\begin{lem}\label{thm:basic-p2}
 Assume $1<p\leq\frac{n^2}{3n-2}$ and $u\in W_{loc}^{1,p}\left(M^n\right)$ is a positive and weak solution to the critical $p$-Laplace equation \eqref{eq:critical-p} on a complete Riemannian manifold $M^n$ with dimension $n$ and nonnegative Ricci curvature. Denote by
\begin{align*}
w=\left(\dfrac{n-p}{p}\right)^{\frac{p-1}{p}}u^{-\frac{p}{n-p}} \quad\quad\mbox{and}\quad\quad f=\dfrac{n(p-1)}{p}w^{-1}\abs{\nabla w}^{p}+w^{-1}.
\end{align*}
We have for every real numbers
$$\alpha<\frac{n(p-1)}{(n-1)p},\quad\quad \mu<\frac{n(p-1)}{p}\quad \mbox{and}\quad \delta\in(0,1]$$
\begin{align}\label{eq:Integral-I}
\int_{B_{R/2}} f^{2-\frac2p-\alpha}w^{3-n-\frac2p}\leq C_{\alpha,\mu,\delta}R^{-\frac{2}{\delta}}\int_{B_{R}} f^{n-1-\alpha-\mu-\frac{\mu+3-n}{\delta}}w^{-\mu+\frac{n-1-\mu}{\delta}},\quad\forall R>0.
\end{align}
\end{lem}

\begin{proof}
Use the same notation as in the proof of \autoref{thm:basic-p}, we know that
\begin{align}\label{eq:Integral-1}
\int f^{-\alpha}w^{1-n}\trace\mathbf{E}^2\eta^{\gamma}\leq &C_{\alpha}\gamma^2\int f^{-\alpha}w^{1-n}\abs{\nabla w}^{2p-2}\abs{\nabla \eta}^{2}\eta^{\gamma-2}.
\end{align}
Notice that
\begin{align*}
\abs{\nabla w}^p\leq Cfw.
\end{align*}
We obtain from \eqref{eq:Integral-1}
\begin{align}\label{eq:Integral-11}
\int f^{-\alpha}w^{1-n}\trace\mathbf{E}^2\eta^{\gamma}\leq &C_{\alpha}\gamma^2\int f^{2-\frac2p-\alpha}w^{3-\frac2p-n}\abs{\nabla \eta}^{2}\eta^{\gamma-2}.
\end{align}

Now we assume
$$1<p\leq\dfrac{n^2}{3n-2}.$$
Notice that this case occurs only when $n\geq3$ and
$$1<p<\dfrac{n+1}{3}.$$

We compute from \eqref{eq:critical-p} for every constants $\beta$ and $\mu$,
\begin{align*}
\int f^{1+\beta}w^{-\mu}\eta^{\gamma}=&\int\left(\Delta_pw\right) f^{\beta}w^{-\mu}\eta^{\gamma}\\
=&-\beta\int f^{\beta-1}\abs{\nabla w}^{p-2}\hin{\nabla w}{\nabla f}w^{-\mu}\eta^{\gamma}+\mu\int f^{\beta}w^{-\mu-1}\abs{\nabla w}^p\eta^{\gamma}-\gamma\int f^{\beta}w^{-\mu}\eta^{\gamma-1}\abs{\nabla w}^{p-2}\hin{\nabla w}{\nabla\eta}\\
=&-n\beta\int f^{\beta-1}\hin{\nabla w}{\mathbf{E}\left(\mathbf{W}\right)}w^{-\mu-1}\eta^{\gamma}+\mu\int f^{\beta}w^{-\mu-1}\abs{\nabla w}^p\eta^{\gamma}-\gamma\int f^{\beta}w^{-\mu}\eta^{\gamma-1}\abs{\nabla w}^{p-2}\hin{\nabla w}{\nabla\eta}\\
\leq&C\abs{\beta}\int f^{\beta-1}\abs{\nabla w}^p\sqrt{\trace\mathbf{E}^2}w^{-\mu-1}\eta^{\gamma}+\mu\int f^{\beta}w^{-\mu-1}\abs{\nabla w}^p\eta^{\gamma}+\gamma\int f^{\beta}w^{-\mu}\eta^{\gamma-1}\abs{\nabla w}^{p-1}\abs{\nabla\eta}\\
\leq&C\abs{\beta}\int f^{\beta}\sqrt{\trace\mathbf{E}^2}w^{-\mu}\eta^{\gamma}+\dfrac{p\mu^+}{n(p-1)}\int f^{1+\beta}w^{-\mu}\eta^{\gamma}+\gamma\int f^{\beta+1-\frac1p}w^{1-\frac1p-\mu}\eta^{\gamma-1}\abs{\nabla\eta}.
\end{align*}
Here we have used \eqref{eq:X} and \eqref{eq:E'}. Using H\"older's inequality with the respective exponent pair $\left(2,2\right)$, it follows that
\begin{align*}
\left(1- \dfrac{p\mu^+}{n(p-1)}\right)\int f^{1+\beta}w^{-\mu}\eta^{\gamma}\leq&C\abs{\beta}\int f^{\beta}\sqrt{\trace\mathbf{E}^2}w^{-\mu}\eta^{\gamma}+\gamma\int f^{\beta+1-\frac1p}w^{1-\frac1p-\mu}\eta^{\gamma-1}\abs{\nabla\eta}\\
\leq&C\abs{\beta}\left(\int f^{-\alpha}w^{1-n}\trace\mathbf{E}^2\eta^{\gamma+2}\right)^{1/2}\left(\int f^{\alpha+2\beta}w^{n-1-2\mu}\eta^{\gamma-2}\right)^{1/2}\\
&+\gamma\left(\int f^{2-\frac2p-\alpha}w^{3-n-\frac2p}\abs{\nabla\eta}^2\eta^{\gamma}\right)^{1/2}\left(\int f^{\alpha+2\beta}w^{n-1-2\mu}\eta^{\gamma-2}\right)^{1/2}.
\end{align*}
Inserting the above inequality into \eqref{eq:Integral-11}, we get
\begin{align*}
\left(1- \dfrac{p\mu^+}{n(p-1)}\right)\int f^{1+\beta}w^{-\mu}\eta^{\gamma}\leq&C_{\alpha,\beta}\gamma\left(\int f^{2-\frac2p-\alpha}w^{3-n-\frac2p}\abs{\nabla\eta}^2\eta^{\gamma}\right)^{1/2}\left(\int f^{\alpha+2\beta}w^{n-1-2\mu}\eta^{\gamma-2}\right)^{1/2}.
\end{align*}
Thus, for
$$\alpha<\dfrac{n(p-1)}{(n-1)p}\quad\quad \mbox{and} \quad\quad \mu<\dfrac{n(p-1)}{p}$$
we have
\begin{align}\label{eq:integral-2}
\int f^{1+\beta}w^{-\mu}\eta^{\gamma}\leq&C_{\alpha,\beta,\mu}\gamma\left(\int f^{2-\frac2p-\alpha}w^{3-n-\frac2p}\abs{\nabla\eta}^2\eta^{\gamma}\right)^{1/2}\left(\int f^{\alpha+2\beta}w^{n-1-2\mu}\eta^{\gamma-2}\right)^{1/2}.
\end{align}

We choose
\begin{align*}
\beta=n-2-\alpha-\mu>n-2-\dfrac{n(p-1)}{(n-1)p}-\dfrac{n(p-1)}{p}=\dfrac{n^2-(3n-2)p}{(n-1)p}\geq 0,
\end{align*}
and obtain
\begin{align*}
3-n-\dfrac{2}{p}+\mu<3-n-\dfrac{2}{p}+\dfrac{n(p-1)}{p}=\dfrac{3p-n-2}{p}<0
\end{align*}
which implies
\begin{align*}
f^{2-\frac2p-\alpha}w^{3-n-\frac2p}=\left(fw\right)^{3-n-\frac2p+\mu} f^{1+\beta}w^{-\mu}\leq f^{1+\beta}w^{-\mu}.
\end{align*}
Here we used the fact
\begin{align*}
 f\geq w^{-1}.
\end{align*}
We choose $\eta\in C_0^{\infty}\left(B_{R}\right)$ satisfying
\begin{align*}
0\leq\eta\leq1,\quad\quad\eta\vert_{B_{R/2}}=1,\quad\quad\abs{\nabla\eta}\leq\dfrac{4}{R}.
\end{align*}
It follows from \eqref{eq:integral-2}
\begin{align}\label{eq:integral-3}
\int f^{1+\beta}w^{-\mu}\eta^{\gamma}\leq&C_{\alpha,\mu}\gamma^2R^{-2}\int f^{\alpha+2\beta}w^{n-1-2\mu}\eta^{\gamma-2}.
\end{align}

For every $\delta\in(0,1]$ and $\gamma>\frac{2}{\delta}$, applying H\"older's inequality, we obtain
\begin{align}\label{eq:integral-4}
\int f^{\alpha+2\beta}w^{n-1-2\mu}\eta^{\gamma-2}\leq\left(\int f^{1+\beta}w^{-\mu}\eta^{\gamma}\right)^{1-\delta}\left(\int f^{1+\beta-\frac{1-\alpha-\beta}{\delta}}w^{-\mu+\frac{n-1-\mu}{\delta}}\eta^{\gamma-\frac{2}{\delta}}\right)^{\delta}.
\end{align}
We conclude from \eqref{eq:integral-3} and \eqref{eq:integral-4} that
\begin{align*}
\int f^{1+\beta}w^{-\mu}\eta^{\gamma}\leq&C_{\alpha,\mu,\delta}\gamma^{\frac{2}{\delta}}R^{-\frac{2}{\delta}}\int f^{1+\beta-\frac{1-\alpha-\beta}{\delta}}w^{-\mu+\frac{n-1-\mu}{\delta}}\eta^{\gamma-\frac{2}{\delta}}.
\end{align*}
Thus
\begin{align*}
 \int f^{2-\frac2p-\alpha}w^{3-n-\frac2p}\eta^{\gamma}\leq&C_{\alpha,\mu,\delta}\gamma^{\frac{2}{\delta}}R^{-\frac{2}{\delta}}\int f^{1+\beta-\frac{1-\alpha-\beta}{\delta}}w^{-\mu+\frac{n-1-\mu}{\delta}}\eta^{\gamma-\frac{2}{\delta}}.
\end{align*}
Consequently, we obtain the desired integral inequality \eqref{eq:Integral-I}.
\end{proof}

We need the following
\begin{lem}\label{lem:crucial-p}
Assume $1<p<n$ and $u\in W_{loc}^{1,p}\left(M^n\right)$ is a positive and weak solution to the critical $p$-Laplace equation \eqref{eq:critical-p} on a complete Riemannian manifold $M^n$ with dimension $n$ and nonnegative Ricci curvature. Denote by
\begin{align*}
w=\left(\dfrac{n-p}{p}\right)^{\frac{p-1}{p}}u^{-\frac{p}{n-p}}\quad\quad\mbox{and}\quad\quad f=\dfrac{n(p-1)}{p}w^{-1}\abs{\nabla w}^{p}+w^{-1}.
\end{align*}
Then for every constants $\theta$ and $a$ satisfying
\begin{align*}
0<\theta\leq1\quad\mbox{and}\quad (p-1)\theta\leq a<\dfrac{(n+1)p-n}{p}-\theta,
\end{align*}
we have
\begin{align*}
\int_{B_R}f^{\theta}w^{-a}\leq C_{a,\theta}R^{n-a-\theta},\quad\forall R>0.
\end{align*}
Moreover, for every
$$0\leq a\leq \frac{(n+1)p-n}{p},$$
there holds
\begin{align*}
\int_{B_R}w^{-a}\leq C_{a}R^{n-a},\quad\forall R>0.
\end{align*}
\end{lem}

\begin{rem}
The proof of the above inequalities for critical $p$-Laplace equation in the Euclidean space $\mathbb{R}^n$ are standard, which can be found in \cite[Lemma 2.1]{Vet24note}. The main idea of the proof can be found in \cite[Lemma 2.4]{SerZou02cauchy} or \cite[Lemma 5.1]{CatMonRon23on} or \cite[Lemma 3.1]{Ou22on}. We include the proof for completeness and also for the convenience of the reader.
\end{rem}

\begin{proof}[Proof of \autoref{lem:crucial-p}]
Without loss of generality, we may assume $M^n$ is noncompact.
\vspace{2ex}

{\bf We consider first the case $\theta=1$.} That is, we prove that: for every constant $q$ satisfying
$$ q<\dfrac{(n+1)p-n}{p}$$
there holds
\begin{align}\label{eq:theta=1}
\int_{B_{R/2}}fw^{1-q}\leq&C_{q}R^{-p}\int_{B_{R}} w^{p-q},\quad\forall R>0.
\end{align}

To demonstrate this, we initially observe in the proof of \autoref{thm:basic-p} that
\begin{align*}
 \Delta_pw=f=\dfrac{n(p-1)}{p}w^{-1}\abs{\nabla w}^{p}+w^{-1}.
\end{align*}
Using the integration by parts, we obtain for each positive number $\gamma\geq p$ and nonnegative function $\eta\in C_0^{\infty}\left(M^n\right)$,
\begin{align*}
 &\dfrac{n(p-1)}{p}\int \abs{\nabla w}^{p}w^{-q}\eta^{\gamma}+\int w^{-q}\eta^{\gamma}\\
 =&\int fw^{1-q}\eta^{\gamma}\\
 =&\int\left(\Delta_pw\right)w^{1-q}\eta^{\gamma}\\
=&(q-1)\int\abs{\nabla w}^pw^{-q}\eta^{\gamma}-\gamma\int\abs{\nabla w}^{p-2}w^{1-q}\hin{\nabla w}{\nabla\eta}\eta^{\gamma-1}.
\end{align*}
Thus
\begin{align}\label{eq:theta=11}
 \int w^{-q}\eta^{\gamma}+\left(\dfrac{(n+1)p-n}{p}-q\right)\int\abs{\nabla w}^pw^{-q}\eta^{\gamma}\leq\gamma\int\abs{\nabla w}^{p-1}w^{1-q}\abs{\nabla\eta}\eta^{\gamma-1}.
\end{align}
The H\"older inequality with the respective exponent pair $\left(p,\frac{p}{p-1}\right)$ yields
\begin{align}\label{eq:theta=12}
 \int\abs{\nabla w}^{p-1}w^{1-q}\abs{\nabla\eta}\eta^{\gamma-1}\leq\left(\int\abs{\nabla w}^pw^{-q}\eta^{\gamma}\right)^{\frac{p-1}{p}}\left(\int w^{p-q}\abs{\nabla\eta}^p\eta^{\gamma-p}\right)^{\frac{1}{p}}.
\end{align}
We obtain from \eqref{eq:theta=11} and \eqref{eq:theta=12}
\begin{align*}
 \int w^{-q}\eta^{\gamma}+\left(\dfrac{(n+1)p-n}{p}-q\right)\int\abs{\nabla w}^pw^{-q}\eta^{\gamma}\leq&\gamma\left(\int\abs{\nabla w}^pw^{-q}\eta^{\gamma}\right)^{\frac{p-1}{p}}\left(\int w^{p-q}\abs{\nabla\eta}^p\eta^{\gamma-p}\right)^{\frac{1}{p}},
\end{align*}
which implies
\begin{align}\label{eq:theta=13}
\int\abs{\nabla w}^pw^{-q}\eta^{\gamma}\leq&C_q\gamma^p\int w^{p-q}\abs{\nabla\eta}^p\eta^{\gamma-p}
\end{align}
and
\begin{align}\label{eq:theta=14}
\int w^{-q}\eta^{\gamma}\leq C_q\gamma^p\int w^{p-q}\abs{\nabla\eta}^p\eta^{\gamma-p}.
\end{align}
Hence, it follows from \eqref{eq:theta=13} and \eqref{eq:theta=14} that for every $q<\frac{(n+1)p-n}{p}$ and $\gamma\geq p$
\begin{align*}
 \int fw^{1-q}\eta^{\gamma}\leq C_q\gamma^p\int w^{p-q}\abs{\nabla\eta}^p\eta^{\gamma-p}.
\end{align*}
Consequently, we obtain the desired integral inequality \eqref{eq:theta=1}.
\vspace{2ex}

{\bf We consider second the case $\theta=0$.} That is, we prove that: if $$0\leq q\leq\dfrac{(n+1)p-n}{p},$$
then
\begin{align}\label{eq:theta=0}
 \int_{B_{R}}w^{-q}\leq C_{q} R^{n-q},\quad\forall R>0.
\end{align}
By H\"older's inequality, it suffices to consider the case
$$p\leq q\leq \dfrac{(n+1)p-n}{p}.$$

Firstly, we consider the case
$$p\leq q<q_0\coloneqq\dfrac{(n+1)p-n}{p}.$$

Applying H\"older's inequality, for $q\geq p$ and $\gamma\geq q$ we have
\begin{align*}
\int w^{p-q}\abs{\nabla\eta}^p\eta^{\gamma-p}\leq\left(\int w^{-q}\eta^{\gamma}\right)^{1-\frac{p}{q}}\left(\int\abs{\nabla\eta}^{q}\eta^{\gamma-q}\right)^{\frac{p}{q}}.
\end{align*}
Consequently, it follows from \eqref{eq:theta=14} and the above inequality that for every $p\leq q<\frac{(n+1)p-n}{p}$ and $\gamma\geq q$
\begin{align*}
\int w^{-q}\eta^{\gamma}\leq C_q\gamma^p\left(\int w^{-q}\eta^{\gamma}\right)^{1-\frac{p}{q}}\left(\int \abs{\nabla\eta}^{q}\eta^{\gamma-q}\right)^{\frac{p}{q}}
\end{align*}
which implies
\begin{align*}
\int w^{-q}\eta^{\gamma}\leq C_q\gamma^{q}\int \abs{\nabla\eta}^{ q}\eta^{\gamma-q}.
\end{align*}
We conclude from the Bishop-Gromov volume comparison theorem that for every $q$ satisfying
$$p\leq q<\dfrac{(n+1)p-n}{p},$$
there holds true
\begin{align*}
\int_{B_{R}} w^{-q}\leq C_{q} R^{n-q}.
\end{align*}

Secondly, we consider the case
$$q=q_0\coloneqq \dfrac{(n+1)p-n}{p}.$$

It follows from \eqref{eq:theta=11}
\begin{align*}
\int w^{-q_0}\eta^{\gamma}\leq\gamma\int\abs{\nabla w}^{p-1}w^{1-q_0}\abs{\nabla\eta}\eta^{\gamma-1}
\end{align*}
which implies for each $\epsilon>0$
\begin{align*}
\int_{B_{R/2}}w^{-q_0}\leq& CR^{-1}\int_{B_R}\abs{\nabla w}^{p-1}w^{1-q_0}\\
\leq&CR^{-1}\left(\int_{B_R}\abs{\nabla w}^{p}w^{\epsilon-q_0}\right)^{\frac{p-1}{p}}\left(\int_{B_{R}}w^{p-\epsilon(p-1)-q_0}\right)^{\frac{1}{p}}\\
\leq&C_{\epsilon}R^{-1}\left(R^{-p}\int_{B_{2R}}w^{p+\epsilon-q_0}\right)^{\frac{p-1}{p}}\left(\int_{B_{R}}w^{p-\epsilon(p-1)-q_0}\right)^{\frac{1}{p}}.
\end{align*}
We choose $\epsilon$ fulfilling
\begin{align*}
0<\epsilon<\min\set{\dfrac{p}{p-1},\,q_0-p}=\min\set{\dfrac{p}{p-1},\, \dfrac{(n-p)(p-1)}{p}}
\end{align*}
to obtain
\begin{align*}
\int_{B_{R/2}}w^{-q_0}\leq C_{\epsilon}R^{-1}\left(R^{-p}R^{n+(p+\epsilon-q_0)}\right)^{\frac{p-1}{p}}
\left(R^{n+(p-\epsilon(p-1)-q_0)}\right)^{\frac{1}{p}}=C_{\epsilon}R^{n-q_0}.
\end{align*}
Here we used the Bishop-Gromov volume comparison theorem. Consequently, we obtain the desired integral inequality \eqref{eq:theta=0}.
\vspace{2ex}

{\bf We consider finally the case $0<\theta<1$.}

Denote by
\begin{align*}
q=\max\set{p,\, \dfrac{a}{\theta}+1-\dfrac{(n+1)p-n}{p}\left(\dfrac{1}{\theta}-1\right)}.
\end{align*}
We claim that
\begin{align*}
p\leq q<\dfrac{(n+1)p-n}{p}\quad\quad\mbox{and}\quad\quad 0\leq\dfrac{a-(q-1)\theta}{1-\theta}\leq\dfrac{(n+1)p-n}{p}.
\end{align*}
In fact, if $(p-1)\theta=a$, then
\begin{align*}
 p=q<\dfrac{(n+1)p-n}{p}\quad\quad\mbox{and}\quad\quad 0=\dfrac{a-(q-1)\theta}{1-\theta}<\dfrac{(n+1)p-n}{p}.
\end{align*}
If $(p-1)\theta<a<\frac{(n+1)p-n}{p}-\theta$, then
\begin{align*}
 p<\dfrac{a}{\theta}+1\quad\quad\mbox{and}\quad\quad \dfrac{a}{\theta}+1-\dfrac{(n+1)p-n}{p}\left(\dfrac{1}{\theta}-1\right)<\dfrac{(n+1)p-n}{p}.
\end{align*}
Thus
\begin{align*}
q=\max\set{p,\, \dfrac{a}{\theta}+1-\dfrac{(n+1)p-n}{p}\left(\dfrac{1}{\theta}-1\right)}<\min\set{\dfrac{(n+1)p-n}{p},\, \dfrac{a}{\theta}+1},
\end{align*}
which implies
\begin{align*}
p\leq q<\dfrac{(n+1)p-n}{p} \quad\quad\mbox{and}\quad\quad 0<\dfrac{a-(q-1)\theta}{1-\theta}\leq \dfrac{(n+1)p-n}{p}.
\end{align*}
Combing \eqref{eq:theta=1} and \eqref{eq:theta=0}, by using the Bishop-Gromov volume comparison theorem, H\"older's inequality gives
\begin{align*}
\int_{B_{R}}f^{\theta}w^{-a}\leq&\left(\int_{B_{R}}w^{1-q}f\right)^{\theta}\left(\int_{B_R}w^{-\frac{a-(q-1)\theta}{1-\theta}}\right)^{1-\theta}\\
\leq&\left(C_{q} R^{-p}\int_{B_{2R}}w^{p-q}\right)^{\theta}\left(\int_{B_R}w^{-\frac{a-(q-1)\theta}{1-\theta}}\right)^{1-\theta}\\
\leq&C_{a,\theta}R^{n-\theta p+\theta(p-q)-(a-(q-1)\theta)}\\
=&C_{a,\theta}R^{n-a-\theta}.
\end{align*}
We complete the proof.
\end{proof}

We also need the following
\begin{lem}\label{lem:crucial-p1}
Let $M^n$ be a complete Riemannnian manifold with dimension $n$ and nonnegative Ricci curvature. If $1<p<n$ and $u$ is a positive and weak solution to \eqref{eq:critical-p} on $M^n$, then for every constants $\theta$ and $a$ satisfying
\begin{align*}
0<\theta<1 \quad\quad\mbox{and}\quad\quad a<(p-1)\theta,
\end{align*}
we have
\begin{align*}
\int_{B_R}f^{\theta}w^{-a}\leq C_{a,\theta}R^{n-\frac{p}{p-1}a},\quad\forall R>1.
\end{align*}
Here
\begin{align*}
w=\left(\dfrac{n-p}{p}\right)^{\frac{p-1}{p}}u^{-\frac{p}{n-p}}\quad\quad\mbox{and}\quad\quad f=\dfrac{n(p-1)}{p}w^{-1}\abs{\nabla w}^{p}+w^{-1}.
\end{align*}
\end{lem}

\begin{proof}
Without loss of generality, we may assume $M^n$ is noncompact. We first claim that
\begin{align*}
 u(x)\geq \left(\min_{\partial B_1}u\right)r(x)^{-\frac{n-p}{p-1}},\quad\forall r(x)\geq1,
\end{align*}
where $r(x)=\mathrm{dist}(x,o)$ is the distance function from $x$ to the fixed point $o$. This is a consequence of the weak comparison principle for weak $p$-superharmonic function and the Laplacian comparison theorem on Riemannian manifolds with nonnegative Ricci curvature. To see this, we consider the function
\begin{align*}
\tilde u(x)=r(x)^{-\frac{n-p}{p-1}}.
\end{align*}
Since the Ricci curvature is nonnegative, the Laplacian comparison theorem implies
\begin{align*}
\Delta r\leq\dfrac{n-1}{r}
\end{align*}
holds in the barrier sense. Thus
\begin{align*}
\Delta_p\tilde u=&\Div\left(\abs{\nabla\tilde u}^{p-1}\nabla\tilde u\right)\\
=&-\left(\dfrac{n-p}{p-1}\right)^{p-1}\Div\left(r^{-(n-1)}\nabla r\right)\\
=&\left(\dfrac{n-p}{p-1}\right)^{p-1}\left((n-1)r^{-n}-r^{1-n}\Delta r\right)\\
\geq&0.
\end{align*}
Since $u$ is positive, we have for every positive number $\epsilon$,
\begin{align*}
u\geq \left(\min_{\partial B_1}u\right)r^{-\frac{n-p}{p-1}}-\epsilon,\quad\forall r>\left(\epsilon\left(\min_{\partial B_1}u\right)\right)^{\frac{p-1}{n-p}}.
\end{align*}
Consequently, the weak comparison principle for the $p$-Laplacian implies
\begin{align*}
u\geq\left(\min_{\partial B_1}u\right)r^{-\frac{n-p}{p-1}},\quad\forall r>1.
\end{align*}
Thus
\begin{align*}
w\leq Cr^{\frac{p}{p-1}}.
\end{align*}

According to \autoref{lem:crucial-p}, we obtain for every $R>1$
\begin{align*}
\int_{B_{R}}f^{\theta}w^{-a}\leq & C_{a,\theta}R^{\frac{p}{p-1}\left((p-1)\theta-a\right)}\int_{B_{R}}f^{\theta}w^{-(p-1)\theta}\\
\leq&C_{a,\theta}R^{\frac{p}{p-1}\left((p-1)\theta-a\right)+n-p\theta}\\
=&C_{a,\theta}R^{n-\frac{p}{p-1}a}.
\end{align*}
We complete the proof.
\end{proof}

\section{A Cheng-Yau type gradient estimate}\label{sec:gradient-estimate}

The aim of this section is to prove the following local Cheng-Yau type gradient estimate.
\begin{theorem}\label{thm:Cheng-Yau} Let $M^n$ be a Riemannian manifold of dimension $n$ and with nonnegative Ricci curvature. Let $u$ be a positive and weak solution to the critical $p$-Laplace equation \eqref{eq:critical-p} in $B_{2R}\subset M^n$. If $1<p<n$, then
\begin{align}\label{eq:local-gradient}
\sup_{B_{R/4}}\abs{\nabla\ln u}\leq C\left(\dfrac{1}{R}+\sup_{B_{R}}u^{\frac{p}{n-p}}\right).
\end{align}
Here the constant $C$ depends only on $n$ and $p$.
\end{theorem}

We will use the Moser iteration to prove this theorem. The proof is obtained by following very closely the one for solutions to subcritical quasilinear elliptic PDEs proposed for \cite{HeSunWan24optimal}. For reader's convenience, we provide a proof here. We will divide the proof into several steps.

In this section, we establish the following notations:
\begin{align*}
&\phi = \ln u, \quad \psi = \abs{\nabla\phi}^p, \quad H = u^{p^2/(n-p)},\\
&\mathbf{\Phi} = \abs{\nabla\phi}^{p-2} \nabla\phi, \quad \mathbf{X} = \nabla\mathbf{\Phi} - \dfrac{\Div\mathbf{\Phi}}{n}g, \quad \mathbf{A}_{\phi} = \mathrm{Id} + (p-2)\abs{\nabla\phi}^{-2}\nabla\phi\otimes\nabla\phi.
\end{align*}
Let $Z = \set{\nabla\phi \neq 0}$ signify the set of critical points of $\phi$, and $\mathcal{L}_{p,\phi}$ represent the linearization of the $p$-Laplacian at $\phi$. Thus, for any function $\xi$, we have
\begin{align*}
 \mathcal{L}_{p,\phi}\xi = \Div\left(\abs{\nabla\phi}^{p-2}\mathbf{A}_{\phi}\left(\nabla\xi\right)\right).
\end{align*}

Firstly, we state a pointwise differential inequality as follows.
\begin{lem}
Let $M^n$ be a Riemannian manifold of dimension $n$ and with nonnegative Ricci curvature. Let $u$ be a positive and weak solution to the critical $p$-Laplace equation \eqref{eq:critical-p} in $B_{2R}\subset M^n$. If $1<p<n$, then the following pointwise inequality holds
\begin{align}\label{eq:diff-ineq}
\begin{split}
\dfrac1p\mathcal{L}_{p,\phi}\psi\geq& \dfrac{(p-1)^2}{n-1}\psi^2+\dfrac{1}{n-1}H^2-\dfrac{(n+1)(p-1)^2+n-1}{(n-1)(n-p)}H\psi\\
&-\abs{\dfrac{2}{(n-1)p}H-\dfrac{(n-3)p+2}{(n-1)p}\psi}\sqrt{p-1}\psi^{-1/p}\sqrt{\hin{\mathbf{A}_{\phi}\left(\nabla \psi\right)}{\nabla \psi} }+\dfrac{n(p-1)}{(n-1)p^2}\psi^{-2/p}\hin{\mathbf{A}_{\phi}\left(\nabla \psi\right)}{\nabla \psi}
\end{split}
\end{align}
provided $\nabla u\neq0$.
\end{lem}

\begin{proof}
Since $u$ satisfies
\begin{align*}
-\Delta_pu=u^{p_S-1}.
\end{align*}
Then $u$ is smooth away from the critical set $Z$ of $u$. We compute in $B_{2R}\setminus Z$
\begin{align*}
\Delta_p\phi=&\Div\left(\abs{\nabla\phi}^{p-2}\nabla\phi\right)\\
=&\Div\left(u^{1-p}\abs{\nabla u}^{p-2}\nabla u\right)\\
=&(1-p)u^{-p}\abs{\nabla u}^p+u^{1-p}\Delta_pu
\end{align*}
which implies
\begin{align}\label{eq:critical-phi}
-\Delta_p\phi=(p-1)\psi+H,
\end{align}
provided $\nabla u\neq0$.

By the definition of the vector field $\mathbf{\Phi}$, we have
\begin{align*}
\Div\mathbf{\Phi}=\Delta_p\phi.
\end{align*}
Direct computation yields
\begin{align}\label{eq:bochner-phi}
\dfrac1p\mathcal{L}_{p,\phi}\psi=&\hin{\nabla_{\nabla_{e_i}\mathbf{\Phi}}\mathbf{\Phi}}{e_i}+\hin{\nabla\Delta_p\phi}{\mathbf{\Phi}}
+Ric\left(\mathbf{\Phi},\mathbf{\Phi}\right).
\end{align}
In fact, on the one hand, it follows from \eqref{eq:bochner-1}
\begin{align}\label{eq:bochner-phi1}
 \Div\nabla_{\mathbf{\Phi}}\mathbf{\Phi}=\sum_{i=1}^n\hin{\nabla_{\nabla_{e_i}\mathbf{\Phi}}\mathbf{\Phi}}{e_i}+\hin{\nabla\Div\mathbf{\Phi}}{\mathbf{\Phi}}+Ric\left(\mathbf{\Phi},\mathbf{\Phi}\right).
\end{align}
On the other hand, it follows from \eqref{eq:WW}
\begin{align}\label{eq:bochner-phi2}
 \nabla_{\mathbf{\Phi}}\mathbf{\Phi}=\dfrac{1}{p}\abs{\nabla\phi}^{p-2}\mathbf{A}_{\phi}\left(\nabla\abs{\nabla \phi}^p\right).
\end{align}
By the definition of $\mathcal{L}_p$, we have
\begin{align}\label{eq:bochner-phi3}
 \mathcal{L}_{p,\phi}\abs{\nabla\phi}^p=\Div\left(\abs{\nabla\phi}^{p-2}\mathbf{A}_{\phi}\left(\nabla\abs{\nabla\phi}^p\right)\right).
\end{align}
Combining with \eqref{eq:bochner-phi1}, \eqref{eq:bochner-phi2} and \eqref{eq:bochner-phi3}, we obtain \eqref{eq:bochner-phi}.

We have the following Kato inequality
\begin{align}\label{eq:Kato-phi}
\sum_{i=1}^n\hin{\nabla_{\nabla_{e_i}\mathbf{\Phi}}\mathbf{\Phi}}{e_i}\geq&\dfrac{1}{n-1}\abs{\Delta_p\phi}^2-\dfrac{2}{(n-1)p}\psi^{-1}\hin{\mathbf{A}_{\phi}\left(\nabla \psi\right)}{\mathbf{\Phi}}\Delta_p\phi+\dfrac{n(p-1)}{(n-1)p^2}\psi^{-2/p}\hin{\mathbf{A}_{\phi}\left(\nabla \psi\right)}{\nabla \psi}
\end{align}
holds in $B_{2R}\setminus Z$. To see this, on the one hand, it follows from the Kato inequality \eqref{eq:Kato}
\begin{align}\label{eq:Kato-phi1}
\trace\mathbf{X}^2\geq\dfrac{n(p-1)}{n-1}\dfrac{\hin{\mathbf{X\left(\Phi\right)}}{\mathbf{A}_{\phi}^{-1}\left(\mathbf{X\left(\Phi\right)}\right)}}{\abs{\mathbf{\Phi}}^2}.
\end{align}
One can check that
\begin{align*}
\mathbf{X\left(\Phi\right)}=&\nabla_{\mathbf{\Phi}}\mathbf{\Phi}-\dfrac{\Div\mathbf{\Phi}}{n}\mathbf{\Phi}\\
=&\dfrac{1}{p}\psi^{1-2/p}\mathbf{A}_{\phi}\left(\nabla\psi\right)-\dfrac{\Delta_p\phi}{n}\mathbf{\Phi}
\end{align*}
which implies
\begin{align}\label{eq:Kato-phi2}
\begin{split}
\hin{\mathbf{X\left(\Phi\right)}}{\mathbf{A}_{\phi}^{-1}\left(\mathbf{X\left(\Phi\right)}\right)}=&\hin{\dfrac{1}{p}\psi^{1-2/p}\mathbf{A}_{\phi}
\left(\nabla\psi\right)-\dfrac{\Delta_p\phi}{n}\mathbf{\Phi}}{\dfrac{1}{p}\psi^{1-2/p}\nabla\psi-\dfrac{\Delta_p\phi}{n(p-1)}\mathbf{\Phi}}\\
=&\dfrac{1}{(p-1)n^2}\abs{\Delta_p\phi}^2\abs{\mathbf{\Phi}}^2+\dfrac{1}{p^2}\psi^{-2/p}\hin{\mathbf{A}_{\phi}\left(\nabla\psi\right)}{\nabla\psi}
\abs{\mathbf{\Phi}}^2\\
&-\dfrac{2}{np(p-1)}\psi^{-1}\hin{\mathbf{A}_{\phi}\left(\nabla \psi\right)}{\mathbf{\Phi}}\Delta_p\phi\abs{\mathbf{\Phi}}^2.
\end{split}
\end{align}
On the other hand, notice that
\begin{align}\label{eq:Kato-phi3}
\sum_{i=1}^n\hin{\nabla_{\nabla_{e_i}\mathbf{\Phi}}\mathbf{\Phi}}{e_i}=\trace\mathbf{X}^2+\dfrac{1}{n}\abs{\Div\mathbf{\Phi}}^2.
\end{align}
We obtain from \eqref{eq:Kato-phi1}, \eqref{eq:Kato-phi2} and \eqref{eq:Kato-phi3}
\begin{align*}
\sum_{i=1}^n\hin{\nabla_{\nabla_{e_i}\mathbf{\Phi}}\mathbf{\Phi}}{e_i}\geq&\dfrac{n(p-1)}{n-1}\left(\dfrac{1}{(p-1)n^2}\abs{\Delta_p\phi}^2+\dfrac{1}{p^2}\psi^{-2/p}\hin{\mathbf{A}_{\phi}\left(\nabla\psi\right)}{\nabla\psi}-\dfrac{2}{np(p-1)}\psi^{-1}\hin{\mathbf{A}_{\phi}\left(\nabla \psi\right)}{\mathbf{\Phi}}\Delta_p\phi\right)\\
&+\dfrac1n\abs{\Delta_p\phi}^2
\end{align*}
which gives the desired Kato inequality \eqref{eq:Kato-phi}.

Since the Ricci curvature is nonnegative, it follows from \eqref{eq:critical-phi}, \eqref{eq:bochner-phi} and \eqref{eq:Kato-phi}
\begin{align*}
\dfrac1p\mathcal{L}_{p,\phi}\psi\geq&\dfrac{1}{n-1}\abs{\Delta_p\phi}^2-\dfrac{2}{(n-1)p}\psi^{-1}\hin{\mathbf{A}_{\phi}\left(\nabla \psi\right)}{\mathbf{\Phi}}\Delta_p\phi+\dfrac{n(p-1)}{(n-1)p^2}\psi^{-2/p}\hin{\mathbf{A}_{\phi}\left(\nabla \psi\right)}{\nabla \psi}+\hin{\nabla\Delta_p\phi}{\mathbf{\Phi}}\\
=&\dfrac{1}{n-1}\abs{(p-1)\psi+H}^2+\dfrac{2}{(n-1)p}\psi^{-1}\hin{\mathbf{A}_{\phi}\left(\nabla \psi\right)}{\mathbf{\Phi}}\left((p-1)\psi+H\right)+\dfrac{n(p-1)}{(n-1)p^2}\psi^{-2/p}\hin{\mathbf{A}_{\phi}\left(\nabla \psi\right)}{\nabla \psi}\\
&-\hin{\nabla\left((p-1)\psi+H\right)}{\mathbf{\Phi}}\\
=&\dfrac{(p-1)^2}{n-1}\psi^2+\dfrac{1}{n-1}H^2-\dfrac{(n+1)(p-1)^2+n-1}{(n-1)(n-p)}H\psi+\hin{\mathbf{A}_{\phi}\left(\nabla \psi\right)}{\left(\dfrac{2}{(n-1)p}H\psi^{-1}-\dfrac{(n-3)p+2}{(n-1)p}\right)\mathbf{\Phi}}\\
&+\dfrac{n(p-1)}{(n-1)p^2}\psi^{-2/p}\hin{\mathbf{A}_{\phi}\left(\nabla \psi\right)}{\nabla \psi}.
\end{align*}
The Cauchy-Schwarz inequality together with the above inequality then gives the desired differential inequality \eqref{eq:diff-ineq}.

\end{proof}

Secondly, we give a local $L^q$ gradient estimate as follows.

\begin{lem}
 Let $M^n$ be a Riemannian manifold of dimension $n$ and with nonnegative Ricci curvature. Let $u$ be a positive and weak solution to the critical $p$-Laplace equation \eqref{eq:critical-p} in $B_{2R}\subset M^n$. If $1<p<n$, then there exists a positive constant $q_1>2p$ such that for all $q\geq q_1$
\begin{align}\label{eq:moser-3}
\norm{\nabla\phi}_{L^{q}\left(B_{R/2}\right)}\leq C\norm{H^{1/p}}_{L^{q}\left(B_{R}\right)}+CqR^{-1}\mathrm{Vol}\left(B_{R}\right)^{\frac{1}{q}}.
\end{align}
\end{lem}

\begin{proof}
 It follows from \eqref{eq:diff-ineq}
\begin{align*}
 \dfrac1p\mathcal{L}_{p,\phi}\psi\geq& \dfrac{(p-1)^2}{2(n-1)}\psi^2-\dfrac{(n+1)(p-1)^2+n-1}{(n-1)(n-p)}H\psi
 +\left(\dfrac{p-1}{p^2}-\dfrac{((n-3)p+2)^2}{2(n-1)p^2(p-1)^2}\right)\psi^{-2/p}\hin{\mathbf{A}_{\phi}\left(\nabla \psi\right)}{\nabla \psi}.
\end{align*}
Notice that for each positive number $\lambda$
\begin{align*}
\dfrac{1}{\lambda}\mathcal{L}_{p,\phi}\psi^{\lambda}=\psi^{\lambda-1}\mathcal{L}_{p,\phi}\psi+(\lambda-1)\psi^{\lambda-1-2/p}\hin{\mathbf{A}_{\phi}\left(\nabla\psi\right)}{\nabla\psi}.
\end{align*}
We get
\begin{align*}
\dfrac{1}{\lambda p}\mathcal{L}_{p,\phi}\psi^{\lambda}\geq &\dfrac{(p-1)^2}{2(n-1)p}\psi^{\lambda+1}-\dfrac{(n+1)(p-1)^2+n-1}{(n-1)(n-p)p}H\psi^{\lambda}\\
&
+\dfrac{1}{p}\left(\lambda-\dfrac1p-\dfrac{((n-3)p+2)^2}{2(n-1)p(p-1)^2}\right)\psi^{\lambda-1-2/p}\hin{\mathbf{A}_{\phi}\left(\nabla \psi\right)}{\nabla \psi}.
\end{align*}
Consequently, we obtain
\begin{align}\label{eq:moser-1}
\dfrac{1}{q_0}\mathcal{L}_{p,\phi}\abs{\nabla\phi}^{q_0}\geq a_0\abs{\nabla\phi}^{q_0+p}-b_0H\abs{\nabla\phi}^{q_0},
\end{align}
where
\begin{align*}
q_0=1+\dfrac{((n-3)p+2)^2}{2(n-1)(p-1)^2},\quad a_0=\dfrac{(p-1)^2}{2(n-1)p}\quad\mbox{and}\quad b_0=\dfrac{(n+1)(p-1)^2+n-1}{(n-1)(n-p)p}.
\end{align*}

For any nonnegative cutoff function $\eta\in C_0^{\infty}\left(M^n\right)$ and positive number $t\geq p-2+q_0$, we consider the test function $$\eta^{2}\abs{\nabla\phi}^{2t+2-p-q_0}$$
and obtain from \eqref{eq:moser-1}
\begin{align*}
&a_0\int \abs{\nabla\phi}^{2t+2}\eta^2-b_0\int H\abs{\nabla\phi}^{2t+2-p}\eta^2\\
\leq&-\dfrac{1}{q_0}\int\abs{\nabla\phi}^{p-2}\hin{\mathbf{A}_{\phi}\left(\nabla\abs{\nabla\phi}^{q_0}\right)}{\nabla\left(\eta^2\abs{\nabla\phi}^{2t+2-p-q_0}\right)}\\
=&-\int\abs{\nabla\phi}^{q_0+p-3}\hin{\mathbf{A}_{\phi}\left(\nabla\abs{\nabla\phi}\right)}{2\eta\abs{\nabla\phi}^{2t+2-p-q_0}\nabla\eta+\left(2t+2-p-q_0\right)\eta^2\abs{\nabla\phi}^{2t+1-p-q_0}\nabla\abs{\nabla\phi}}\\
=&-\dfrac{2t+2-p-q_0}{t^2}\int\hin{\mathbf{A}_{\phi}\left(\nabla\abs{\nabla\phi}^{t}\right)}{\nabla\abs{\nabla\phi}^{t}}\eta^2-\dfrac{2}{t}\int\hin{\mathbf{A}_{\phi}\left(\nabla\abs{\nabla\phi}^{t}\right)}{\nabla\eta}\eta \abs{\nabla\phi}^{t}\\
=&-\dfrac{2t+2-p-q_0}{t^2}\int\hin{\mathbf{A}_{\phi}\left(\nabla\left(\eta\abs{\nabla\phi}^{t}\right)\right)}{\nabla\left(\eta\abs{\nabla\phi}^{t}\right)}+\dfrac{p+q_0-2}{t^2}\int\abs{\nabla\phi}^{2t}\hin{\mathbf{A}_{\phi}\left(\nabla\eta\right)}{\nabla\eta}\\
 &+\dfrac{2\left(t+2-p-q_0\right)}{t^2}\int\hin{\mathbf{A}_{\phi}\left(\nabla\left(\eta\abs{\nabla\phi}^{t}\right)\right)}{\nabla\eta} \abs{\nabla\phi}^{t}\\
\leq&-\dfrac{2t+2-p-q_0}{t^2}\int\hin{\mathbf{A}_{\phi}\left(\nabla\left(\eta\abs{\nabla\phi}^{t}\right)\right)}
{\nabla\left(\eta\abs{\nabla\phi}^{t}\right)}+\dfrac{p+q_0-2}{t^2}\int\abs{\nabla\phi}^{2t}\hin{\mathbf{A}_{\phi}\left(\nabla\eta\right)}{\nabla\eta}\\
&+\dfrac{2\left(t+2-p-q_0\right)}{t^2}\left(\int\hin{\mathbf{A}_{\phi}\left(\nabla\left(\eta\abs{\nabla\phi}^{t}\right)\right)}{\nabla\left(\eta\abs{\nabla\phi}^{t}\right)}\right)^{1/2}\left(\int\abs{\nabla\phi}^{2t}\hin{\mathbf{A}_{\phi}\left(\nabla\eta\right)}{\nabla\eta}\right)^{1/2}\\
\leq&-\dfrac{1}{t}\int\hin{\mathbf{A}_{\phi}\left(\nabla\left(\eta\abs{\nabla\phi}^{t}\right)\right)}{\nabla\left(\eta\abs{\nabla\phi}^{t}\right)}
+\dfrac{1}{t}\int\abs{\nabla\phi}^{2t}\hin{\mathbf{A}_{\phi}\left(\nabla\eta\right)}{\nabla\eta}.
\end{align*}
Thus, by observing the eigenvalues of $\mathbf{A}_{\phi}$ are bounded from below by $1$ and above by $p-1$, we obtain
\begin{align}\label{eq:moser-2}
 \dfrac{p-1}{tp}\int\abs{\nabla\left(\eta\abs{\nabla\phi}^{t}\right)}^2+a_0\int \abs{\nabla\phi}^{2t+2}\eta^2\leq b_0\int H\abs{\nabla\phi}^{2t+2-p}\eta^2+\dfrac{p}{t}\int\abs{\nabla\phi}^{2t}\abs{\nabla\eta}^2.
\end{align}

Replaying $\eta$ by $\eta^{t+1}$, we get
\begin{align*}
 a_0\int \abs{\nabla\phi}^{2t+2}\eta^{2t+2}\leq& b_0\int H\abs{\nabla\phi}^{2t+2-p}\eta^{2t+2}+\dfrac{(t+1)^2p}{t}\int\abs{\nabla\phi}^{2t}\eta^{2t}\abs{\nabla\eta}^{2}\\
 \leq&b_0\left(\int\abs{\nabla\phi}^{2t+2}\eta^{2t+2}\right)^{1-\frac{p}{2t+2}}\left(\int H^{\frac{2t+2}{p}}\eta^{2t+2}\right)^{\frac{p}{2t+2}}+\dfrac{(t+1)^2p}{t}\left(\int\abs{\nabla\phi}^{2t+2}\eta^{2t+2}\right)^{\frac{t}{t+1}}\left(\int\abs{\nabla\eta}^{2t+2}\right)^{\frac{1}{t+1}}.
\end{align*}
Applying Young's inequality, we obtain
\begin{align*}
 a_0\int \abs{\nabla\phi}^{2t+2}\eta^{2t+2}\leq& b_0\left(\int\abs{\nabla\phi}^{2t+2}\eta^{2t+2}\right)^{1-\frac{p}{2t+2}}\left(\int H^{\frac{2t+2}{p}}\eta^{2t+2}\right)^{\frac{p}{2t+2}}+\dfrac{(t+1)^2p}{t}\left(\int\abs{\nabla\phi}^{2t+2}\eta^{2t+2}\right)^{\frac{t}{t+1}}\left(\int\abs{\nabla\eta}^{2t+2}\right)^{\frac{1}{t+1}}\\
 \leq&\left(1-\dfrac{p}{2(t+1)}\right)a_0\int \abs{\nabla\phi}^{2t+2}\eta^{2t+2}+\dfrac{p}{2(t+1)}\left(\dfrac{b_0}{a_0}\right)^{\frac{2(t+1)}{p}}a_0\int H^{\frac{2t+2}{p}}\eta^{2t+2}\\
 &+\dfrac{t}{t+1}\dfrac{pa_0}{4t}\int \abs{\nabla\phi}^{2t+2}\eta^{2t+2}+\dfrac{1}{t+1}\left(\dfrac{4(t+1)^2}{a_0}\right)^{t+1}\dfrac{pa_0}{4t}\int\abs{\nabla\eta}^{2t+2}
\end{align*}
which gives
\begin{align*}
 \int \abs{\nabla\phi}^{2t+2}\eta^{2t+2}\leq 2\left(\dfrac{b_0}{a_0}\right)^{\frac{2(t+1)}{p}}\int H^{\frac{2t+2}{p}}\eta^{2t+2}+\left(\dfrac{4(t+1)^2}{a_0}\right)^{t+1}\dfrac1t\int\abs{\nabla\eta}^{2t+2}.
\end{align*}
We complete the proof and obtain the local $L^q$ gradient estimate \eqref{eq:moser-3} for $q_1=2(p-1+q_0)$.

\end{proof}

Ultimately, we can employ the Moser iteration technique to demonstrate the $L^{\infty}$ gradient estimate as detailed below. We shall now present a

\begin{proof}[Proof of \autoref{thm:Cheng-Yau}]
We shall establish this theorem for the scenario where $n\geq3$. The proof when $n=2$ parallels that of the general case $n\geq3$.

Without loss of generality, we may assume
\begin{align*}
 \sup_{B_{R/4}}\abs{\nabla\ln u}>\left(\dfrac{b_0}{a_0}\right)^{1/p}\sup_{B_{R}}u^{\frac{p}{n-p}}.
\end{align*}

We choose the sequences of $R_i$ and $q_i$ such that
\begin{align*}
R_i=\dfrac{R}{4}\left(1+\dfrac{1}{2^i}\right),\quad q_i=\tilde q_0\chi^i,\quad i=0,1,2,\dots.
\end{align*}
Here $\chi=\frac{n}{n-2}$ and
\begin{align*}
 \tilde q_0=2t_0+\norm{H}_{L^{\infty}\left(B_{R}\right)}^{1/p}R.
\end{align*}
For each $i$, we can choose a cutoff function $\eta_i\in C_0^{\infty}\left(B_{R_i}\right)$ satisfying
\begin{align*}
0\leq\eta_i\leq 1,\quad \eta_i\vert_{B_{R_{i+1}}}=1,\quad \abs{\nabla\eta_i}\leq 2^{i+4}/R.
\end{align*}
We claim that there exist a positive integer number $i_0$ such that for all $i\geq i_0$
\begin{align}\label{eq:moser-4}
 a_0\int_{B_{R_i}}\abs{\nabla\phi}^{q_i+2}\eta_i^2\geq b_0\int_{B_{R_i}} H\abs{\nabla\phi}^{q_i+2-p}\eta_i^2.
\end{align}
For otherwise, there exist a subsequence $i_k$ of positive integer numbers such that $i_k\to\infty$ as $k\to\infty$ and
\begin{align*}
a_0\int_{B_{R_{i_k}}}\abs{\nabla\phi}^{q_{i_k}+2}\eta_{i_k}^2<b_0\int_{B_{R_{i_k}}} H\abs{\nabla\phi}^{q_{i_k}+2-p}\eta_{i_k}^2,\quad \forall k.
\end{align*}
Applying H\"older's inequality,
\begin{align*}
 \int_{B_{R_{i_k}}} H\abs{\nabla\phi}^{q_{i_k}+2-p}\eta_{i_k}^2<& \left(\int_{B_{R_{i_k}}}\abs{\nabla\phi}^{q_{i_k}+2}\eta_{i_k}^2\right)^{\frac{q_{i_k}+2-p}{q_{i_k}+2}}\left(\int_{B_{R_{i_k}}} H^{\frac{q_{i_k}+2}{p}}\eta_{i_k}^2\right)^{\frac{p}{q_{i_k}+2}}.
\end{align*}
We obtain
\begin{align*}
 \left(\int_{B_{R_{i_k}}}\abs{\nabla\phi}^{q_{i_k}+2}\eta_{i_k}^2\right)^{\frac{1}{q_{i_k}+2}}<\left(\dfrac{b_0}{a_0}\right)^{\frac1p}\left(\int_{B_{R_{i_k}}} H^{\frac{q_{i_k}+2}{p}}\eta_{i_k}^2\right)^{\frac{1}{q_{i_k}+2}}
\end{align*}
which implies
\begin{align*}
 \left(\int_{B_{R_{i_k+1}}}\abs{\nabla\phi}^{q_{i_k}+2}\right)^{\frac{1}{q_{i_k}+2}}<\left(\dfrac{b_0}{a_0}\right)^{\frac1p}\left(\int_{B_{R_{i_k}}} H^{\frac{q_{i_k}+2}{p}}\right)^{\frac{1}{q_{i_k}+2}}.
\end{align*}
We obtain by letting $k\to\infty$
\begin{align*}
 \norm{\nabla\ln u}_{L^{\infty}\left(B_{R/4}\right)}=\norm{\nabla\phi}_{L^{\infty}\left(B_{R/4}\right)}\leq\left(\dfrac{b_0}{a_0}\right)^{\frac1p}\norm{H^{1/p}}_{L^{\infty}\left(B_{R/4}\right)}=\left(\dfrac{b_0}{a_0}\right)^{\frac1p}\norm{u^{\frac{p}{n-p}}}_{L^{\infty}\left(B_{R/4}\right)}\leq \left(\dfrac{b_0}{a_0}\right)^{\frac1p}\norm{u^{\frac{p}{n-p}}}_{L^{\infty}\left(B_{R}\right)}
\end{align*}
which is a contradiction.

Recall Saloff-Coste's Dirichlet Sobolev inequality \cite[Theorem 3.1]{Saloff92uniformly}: for every $v\in C_0^{\infty}\left(B_{R}\right)$
\begin{align*}
\left(\fint_{B_{R}}\abs{v}^{\frac{2n}{n-2}}\right)^{\frac{n-2}{n}}\leq CR^2\left(\fint_{B_{R}}\abs{\nabla v}^2+R^{-2}\fint_{B_{R}}\abs{v}^2\right).
\end{align*}
By applying the above Sobolev inequality, it follows from \eqref{eq:moser-2} and \eqref{eq:moser-4}
\begin{align*}
 \dfrac{1}{CR^2}\left(\int_{B_{R_i}}\left(\eta_i\abs{\nabla\phi}^{q_i/2}\right)^{\frac{2n}{n-2}}\right)^{\frac{n-2}{n}}-\dfrac{1}{R^2}V_R^{-\frac2n}\int_{B_{R_i}}\left(\eta_i\abs{\nabla\phi}^{q_i/2}\right)^2\leq \dfrac{p^2}{p-1}V_R^{-\frac2n}\int_{B_{R_i}}\abs{\nabla\phi}^{q_i}\abs{\nabla\eta_i}^2,\quad\forall i\geq i_0.
\end{align*}
Here $V_{R}=\mathrm{Vol}\left(B_R\right)$. Thus
\begin{align*}
 \left(\int_{B_{R_{i+1}}}\abs{\nabla\phi}^{q_{i+1}}\right)^{\frac{1}{\chi}}\leq C4^iV_{R}^{-\frac2n}\int_{B_{R_i}}\abs{\nabla\phi}^{q_i},\quad\forall i\geq i_0
\end{align*}
which gives
\begin{align*}
 \norm{\nabla\phi}_{L^{q_{i+1}}\left(B_{R_{i+1}}\right)}\leq \left(C4^iV_{R}^{-\frac2n}\right)^{\frac{1}{\tilde q_0\chi^i}}\norm{\nabla\phi}_{L^{q_{i}}\left(B_{R_{i}}\right)},\quad \forall i\geq i_0.
\end{align*}
Without loss of generality, assume $i_0=0$.
Notice that
\begin{align*}
 \Pi_{i=0}^{\infty}\left(C4^iV_{R}^{-\frac2n}\right)^{\frac{1}{\tilde q_0\chi^i}}=&C^{\frac{n}{2\tilde q_0}}4^{\frac{\chi}{\tilde q_0\left(\chi-1\right)^2}}V_{R}^{-\frac{1}{\tilde q_0}}.
\end{align*}
We obtain
\begin{align*}
 \norm{\nabla\phi}_{L^{\infty}\left(B_{R/4}\right)}\leq&C^{\frac{n}{2\tilde q_0}}4^{\frac{\chi}{\tilde q_0\left(\chi-1\right)^2}}V_{R}^{-\frac{1}{\tilde q_0}}\norm{\nabla\phi}_{L^{\tilde q_0}\left(B_{R/2}\right)}.
\end{align*}
In view of the $L^q$ gradient estimate \eqref{eq:moser-3}, we get
\begin{align*}
 \norm{\nabla\phi}_{L^{\infty}\left(B_{R/4}\right)}\leq&C^{\frac{n}{2\tilde q_0}}4^{\frac{\chi}{\tilde q_0\left(\chi-1\right)^2}}V_{R}^{-\frac{1}{\tilde q_0}}C\left(\norm{H^{1/p}}_{L^{\tilde q_0}\left(B_{R}\right)}+\tilde q_0R^{-1}V_R^{\frac{1}{\tilde q_0}}\right)\\
 \leq&C\left(\norm{H}^{1/p}_{L^{\infty}\left(B_{R}\right)}+\tilde q_0R^{-1}\right)\\
 \leq&C\left(R^{-1}+\norm{H}_{L^{\infty}\left(B_{R}\right)}^{1/p}\right).
\end{align*}
We complete the proof.
\end{proof}

\section{Rigidity results}\label{sec:rigidity}
\subsection{Rigidity results for higher dimension}

In this subsection, we will give a proof of the first part of \autoref{thm:main1}. That is, we will prove
\begin{theorem}\label{thm:large-p}
Let $M^n$ be a complete, connected and noncompact Riemannian manifold $M^n$ of dimension $n$ and with nonnegative Ricci curvature. Denote by \begin{align*}
p_n=\begin{cases}
\frac{n^2}{3n-2},&n=2,3,4,\\
\frac{n^2+2}{3n},&n\geq5.
\end{cases}
\end{align*}
If $p_n<p<n$ and there is a positive and weak solution $u$ to the critical $p$-Laplace equation \eqref{eq:critical-p} in $M^n$, then $M^n$ is isometric to the Euclidean space $\mathbb{R}^n$ and the solution is given by
\begin{align*}
u(x)=\left(a+b\abs{x-x_0}^{\frac{p}{p-1}}\right)^{-\frac{n-p}{p}}
\end{align*}
for some $x_0\in\mathbb{R}^n$ and $a, b>0$ satisfying $n\left(\frac{n-p}{p-1}\right)^{p-1}ab^{p-1}=1$.
\end{theorem}

\begin{proof}
Denote by
\begin{align*}
w=\left(\dfrac{n-p}{p}\right)^{\frac{p-1}{p}}u^{-\frac{p}{n-p}}\quad\mbox{and}\quad f=\dfrac{n(p-1)}{p}w^{-1}\abs{\nabla w}^{p}+w^{-1}.
\end{align*}
We claim that there exists a positive constant $\epsilon_1$ such that for all $0<\epsilon<\epsilon_1$,
\begin{align*}
\int_{B_R}w^{3-n-\frac2p}f^{\frac{(n-2)(p-1)}{(n-1)p}+\epsilon}\leq C_{\epsilon}R^{2-\epsilon},\quad\forall R>0.
\end{align*}

{\bf Firstly, we consider the case
$$\dfrac{n^2}{3n-2}<p<n.$$}

Denote by
\begin{align*}
a=n+\frac2p-3 \quad\quad\mbox{and}\quad\quad \theta=\frac{(n-2)(p-1)}{(n-1)p}+\epsilon.
\end{align*}
Then
\begin{align*}
&1-\theta=\dfrac{p+n-2}{(n-1)p}-\epsilon,\\
&\dfrac{(n+1)p-n}{p}-a-\theta=\dfrac{(3n-2)p-n^2}{(n-1)p}-\epsilon,\\
&a-(p-1)\theta=\dfrac{((n-2)p+1)(n-p)}{(n-1)p}-(p-1)\epsilon.
\end{align*}
Thus, for every
\begin{align*}
0<\epsilon<\epsilon_1'\coloneqq\min\set{\dfrac{p+n-2}{(n-1)p},\,\dfrac{(3n-2)p-n^2}{(n-1)p},\, \dfrac{((n-2)p+1)(n-p)}{(n-1)p(p-1)}},
\end{align*}
we have
\begin{align*}
\theta\in(0,\,1),\quad\quad (p-1)\theta< a<\dfrac{(n+1)p-n}{p}-\theta,
\end{align*}
and
\begin{align*}
n-a-\theta=2-\dfrac{n-p}{(n-1)p}-\epsilon.
\end{align*}
According to \autoref{lem:crucial-p}, we conclude
\begin{align*}
\int_{B_R}f^{\frac{(n-2)(p-1)}{(n-1)p}+\epsilon}w^{3-n-\frac2p}=&\int_{B_{R}}f^{\theta}w^{-a}\\
\leq& C_{a,\theta}R^{n-a-\theta}\\
=&C_{\epsilon}R^{2-\frac{n-p}{(n-1)p}-\epsilon}.
\end{align*}
\vspace{2ex}

{\bf Secondly, we consider the case $n\geq5$ and
$$\dfrac{n^2+2}{3n}<p\leq\dfrac{n^2}{3n-2}.$$}

Denote by
\begin{align*}
\alpha=\dfrac{n(p-1)}{(n-1)p}-\epsilon,\quad\mu=\dfrac{n(p-1)}{p}-\epsilon.
\end{align*}
According to \autoref{thm:basic-p2}, for every $\delta\in(0,1]$
\begin{align*}
\int_{B_{R/2}}w^{3-n-\frac2p}f^{\frac{(n-2)(p-1)}{(n-1)p}+\epsilon}=\int_{B_{R/2}}w^{3-n-\frac2p}f^{2-\frac2p-\alpha}\leq C_{\epsilon,\delta}R^{-\frac{2}{\delta}}\int_{B_{R}} f^{\theta'}w^{-a'},\quad\forall R>0,
\end{align*}
where
\begin{align*}
\theta'\coloneqq& n-1-\alpha-\mu-\dfrac{\mu+3-n}{\delta}=\dfrac{n^2-(2n-1)p}{(n-1)p}-\dfrac{3p-n}{\delta p}+\left(2+\dfrac1\delta\right)\epsilon,\\
a'\coloneqq&\mu-\dfrac{n-1-\mu}{\delta}=\dfrac{n(p-1)}{p}-\dfrac{n-p}{\delta p}-\left(1+\dfrac1\delta\right)\epsilon.
\end{align*}
Direct computation gives
\begin{align*}
&1-\theta'=\dfrac{3p-n}{\delta p}-\dfrac{n^2-(3n-2)p}{(n-1)p}-\left(2+\dfrac1\delta\right)\epsilon,\\
&\dfrac{p}{p-1}a'+\dfrac{2}{\delta}-n+2=2-\dfrac{n+2-3p}{(p-1)\delta}-\dfrac{p}{p-1}\left(1+\dfrac1\delta\right)\epsilon,
\end{align*}
and
\begin{align*}
1-\theta'=&\left(\dfrac{(n+1)p-n}{p}-a'-\theta'\right)-\left(\dfrac{n-p}{\delta p}+\left(1+\dfrac1\delta\right)\epsilon\right).
\end{align*}
Since $n\geq5$ and
$$\dfrac{n^2+2}{3n}<p\leq\dfrac{n^2}{3n-2},$$
one can check
\begin{align*}
\max\set{\dfrac{(3p-n)(n-1)}{n^2-(2n-1)p},\,\dfrac{n+2-3p}{2(p-1)}}<\min\set{1,\,\dfrac{(n-1)(3p-n)}{n^2-(3n-2)p}}.
\end{align*}
Thus, we can choose $\delta$ satisfying
\begin{align*}
\max\set{\dfrac{(3p-n)(n-1)}{n^2-(2n-1)p},\,\dfrac{n+2-3p}{2(p-1)}}<\delta<\min\set{1,\,\dfrac{(n-1)(3p-n)}{n^2-(3n-2)p}}
\end{align*}
which implies for some small positive number $\epsilon_1''<\epsilon_1'$ such that for all $0<\epsilon<\epsilon_1''$
\begin{align*}
&\delta\in(0,1),\quad \theta'\in(0,1),\quad a'+\theta'<\dfrac{(n+1)p-n}{p},\\
&n-a'-\theta'-\dfrac{2}{\delta}=2-\dfrac{n-p}{(n-1)p}-\epsilon<2-\epsilon,\quad n-\dfrac{p}{p-1}a'-\dfrac{2}{\delta}<2-\epsilon.
\end{align*}
Applying \autoref{lem:crucial-p} and \autoref{lem:crucial-p1}, we get
\begin{align*}
\int_{B_R}w^{3-n-\frac2p}f^{\frac{(n-2)(p-1)}{(n-1)p}+\epsilon}\leq C_{\epsilon}R^{\max\set{n-a'-\theta'-\frac{2}{\delta},\, n-\frac{p}{p-1}a'-\frac{2}{\delta}}}\leq C_{\epsilon}R^{2-\epsilon},\quad\forall R>1.
\end{align*}
\vspace{2ex}

Now we can continue the proof as follows. If $f$ is not a constant function, then according to \autoref{thm:basic-p}, we know that for each real number
$$\alpha<\dfrac{n(p-1)}{(n-1)p}$$
there holds
\begin{align*}
\liminf_{R\to\infty}\dfrac{1}{R^2}\int_{B_{R}}f^{-\alpha}w^{1-n}\abs{\nabla w}^{2p-2}=\infty.
\end{align*}
In particular, for every $\epsilon\in\left(0,\epsilon_1\right)$
\begin{align*}
 C_{\epsilon}^{-1}R^2\leq\int_{B_R}w^{3-n-\frac2p}f^{\frac{(n-2)(p-1)}{(n-1)p}+\epsilon}\leq C_{\epsilon}R^{2-\epsilon},\quad\forall R>0
\end{align*}
which is impossible. Therefore, $f$ is a constant function. It follows from \eqref{eq:X} that $\mathbf{E\left(W\right)}\equiv0$. The Bochner formula \eqref{eq:bochner} then yields
\begin{align*}
 \trace\mathbf{E}^2\equiv0 \quad\mbox{and}\quad Ric\left(\mathbf{W},\mathbf{W}\right)\equiv0.
\end{align*}
According to \eqref{eq:E}, we know that $\mathbf{E}\equiv0$. In other words,
\begin{align*}
 \nabla_X\mathbf{W}=\dfrac{\Div\mathbf{W}}{n}X,\quad\forall X\in TM.
\end{align*}
Notice that $\Div\mathbf{W}=\Delta_pw=f$ is a constant function. We conclude that $\mathbf{W}$ is a nontrivial homothetic vector field, i.e., $\mathbf{W}\neq0$ and satisfies
\begin{align*}
 \nabla\mathbf{W}=c g
\end{align*}
for some positive constant $c$. The function $h=\frac{1}{2}\abs{\mathbf{W}}^2$ then satisfies
\begin{align*}
 \nabla h= c\mathbf{W}.
\end{align*}
Thus
\begin{align*}
 \nabla^2h=c^2g.
\end{align*}
We conclude that $M^n$ is isometric to the Euclidean space $\mathbb{R}^n$ (cf. \cite[Theorem 2 (I, B)]{Tas65complete}).
Finally, by the classification result for the Euclidean case (see e.g. \cite{CafGidSpr89asymptotic, DamMerMonSci14radial, Sci16classification,Vet16priori}), we know that
 the solution is give by
\begin{align*}
 u(x)=\left(a+b\abs{x-x_0}^{\frac{p}{p-1}}\right)^{-\frac{n-p}{p}}
\end{align*}
for some $x_0\in\mathbb{R}^n$ and $a>0$, $b>0$ satisfying
$$n\left(\frac{n-p}{p-1}\right)^{p-1}ab^{p-1}=1.$$
\end{proof}

\subsection{Rigidity results for solutions with finite energy}
In this subsection, we will give a proof of the second part of \autoref{thm:main1}. In other words, we will prove

\begin{theorem}\label{thm:finite}
Let $M^n$ be a complete, connected and noncompact Riemannian manifold $M^n$ of dimension $n$ and with nonnegative Ricci curvature. If $1<p<n$ and there is a positive and weak solution $u$ to the critical $p$-Laplace equation \eqref{eq:critical-p} satisfying
\begin{align}\label{eq:energy}
\limsup_{R\to\infty} R^{q-n}\int_{B_{R}}u^{\frac{pq}{n-p}}<\infty
\end{align}
for some constant
$$q>\dfrac{(n^2-3n+1)p+n}{(n-1)p},$$
then $M^n$ is isometric to the Euclidean space $\mathbb{R}^n$ and the solution is given by
\begin{align*}
u(x)=\left(a+b\abs{x-x_0}^{\frac{p}{p-1}}\right)^{-\frac{n-p}{p}}
\end{align*}
for some $x_0\in\mathbb{R}^n$ and $a, b>0$ satisfying
$$n\left(\dfrac{n-p}{p-1}\right)^{p-1}ab^{p-1}=1.$$
\end{theorem}

\begin{proof}
According to \autoref{thm:large-p}, without loss of generality, we may assume \
$$1<p\leq\frac{n^2}{3n-2}.$$
Denote by
\begin{align*}
w=\left(\dfrac{n-p}{p}\right)^{\frac{p-1}{p}}u^{-\frac{p}{n-p}} \quad\quad\mbox{and}\quad\quad f=\dfrac{n(p-1)}{p}w^{-1}\abs{\nabla w}^{p}+w^{-1}.
\end{align*}
We claim that there exists a positive constant $\epsilon_2$ such that for all $0<\epsilon<\epsilon_2$,
\begin{align*}
\int_{B_{R}} f^{\frac{(n-2)(p-1)}{(n-1)p}+\epsilon}w^{3-\frac2p-n}\leq C_{\epsilon}R^{2-\frac{n-p}{(n-1)p}-\epsilon}.
\end{align*}

We claim first that the following three properties are equivalent to each other:
\begin{enumerate}[(E1)]
 \item For every $R>0$
 \begin{align*}
 \int_{B_{R}}w^{q}\leq C_{q}R^{n-q}.
 \end{align*}
 \item For every $R>0$
 \begin{align*}
 \int_{B_{R}}\abs{\nabla w}^pw^{-q}\leq C_{q}R^{n-q}.
 \end{align*}
 \item For every $R>0$
 \begin{align*}
 \int_{B_{R}}fw^{1-q}\leq C_{q}R^{n-q}.
 \end{align*}
\end{enumerate}

By observing
\begin{align*}
 \Delta_pw=f,
\end{align*}
we have for every nonnegative function $\eta\in C_0^{\infty}\left(M^n\right)$,
\begin{align*}
\int fw^{1-q}\eta^{q}=&\int\left(\Delta_pw\right)w^{1-q}\eta^{q}\\
=&\left(q-1\right)\int w^{-q}\abs{\nabla w}^{p}\eta^{q}-q\int w^{1-q}\abs{\nabla w}^{p-2}\hin{\nabla w}{\nabla\eta}\eta^{q-1}.
\end{align*}
That is
\begin{align*}
\left(q-\dfrac{(n+1)p-n}{p}\right)\int w^{-q}\abs{\nabla w}^{p}\eta^{q}-\int w^{-q}\eta^{q}=q\int w^{1-q}\abs{\nabla w}^{p-2}\hin{\nabla w}{\nabla\eta}\eta^{q-1}.
\end{align*}
Since
$$1<p\leq\frac{n^2}{3n-2},$$
we have
\begin{align*}
q>\dfrac{(n^2-3n+1)p+n}{(n-1)p}\geq\dfrac{(n+1)p-n}{p}>p>1
\end{align*}
which implies
\begin{align*}
1>\dfrac{1}{p}>\dfrac{1}{q}>0.
\end{align*}
By applying H\"older's inequality with the respective exponent triplet $\left(\frac{p}{p-1},\,\frac{pq}{q-p},\,q\right)$, we obtain
\begin{align*}
\abs{\left(q-\dfrac{(n+1)p-n}{p}\right)\int w^{-q}\abs{\nabla w}^{p}\eta^{q}-\int w^{-q}\eta^{q}}\leq q\left(\int w^{-q}\abs{\nabla w}^{p}\eta^{q}\right)^{1-\frac{1}{p}}\left(\int w^{-q}\eta^{q}\right)^{\frac1p-\frac{1}{q}}\left(\int\abs{\nabla\eta}^{q}\right)^{\frac{1}{q}}.
\end{align*}
Utilizing Young's inequality with the respective exponent triplet $\left(\frac{p}{p-1},\,\frac{pq}{q-p},\,q\right)$, we derive
\begin{align*}
C_q^{-1}\left( \int w^{-q}\abs{\nabla w}^{p}\eta^{q}+\int\abs{\nabla\eta}^{q}\right)\leq \int w^{-q}\eta^{q}+\int\abs{\nabla\eta}^{q}\leq C_q\left( \int w^{-q}\abs{\nabla w}^{p}\eta^{q}+\int\abs{\nabla\eta}^{q}\right),
\end{align*}
where $C_q$ is some positive constant depending only on $n$, $p$ and $q$. We choose $\eta\in C_0^{\infty}\left(B_{R}\right)$ with the properties
\begin{align*}
0\leq\eta\leq1,\quad \eta\vert_{B_{R/2}}=1,\quad \abs{\nabla\eta}\leq\dfrac{4}{R}.
\end{align*}
According to the Bishop-Gromov volume comparison theorem, we know that the three properties $(E1)-(E3)$ mentioned above are equivalent to each other.

\vspace{2ex}

Now we assume the property $(E1)$ holds, i.e., the assumption \eqref{eq:energy} holds.

By H\"older's inequality, without loss of generality, we may assume
\begin{align*}
 \dfrac{(n^2-3n+1)p+n}{(n-1)p}<q<\dfrac{(n^2-3n+1)p+n}{(n-2)(p-1)}.
\end{align*}
Take
\begin{align*}
 \alpha=\dfrac{n(p-1)}{(n-1)p}-\epsilon,
\end{align*}
where $\epsilon$ is a small positive number to be determined.

Denote by
\begin{align*}
\theta\coloneqq 2-\dfrac2p-\alpha=\dfrac{(n-2)(p-1)}{(n-1)p}+\epsilon \quad\quad\mbox{and}\quad\quad a\coloneqq &n+\dfrac2p-3.
\end{align*}
We have
\begin{align*}
&a-(q-1)\theta=\dfrac{(n^2-3n+1)p+n}{(n-1)p}-\dfrac{q(n-2)(p-1)}{(n-1)p}-(q-1)\epsilon,\\
&q-a-\theta=q-\dfrac{(n^2-3n+1)p+n}{(n-1)p}-\epsilon.
\end{align*}
Thus, for every
\begin{align*}
0<\epsilon<\epsilon_2\coloneqq\min\set{q-\dfrac{(n^2-3n+1)p+n}{(n-1)p},\,\, \dfrac{n+2-p}{(n-1)p}, \, \dfrac{1}{q-1}\left(\dfrac{(n^2-3n+1)p+n}{(n-1)p}-\dfrac{q(n-2)(p-1)}{(n-1)p}\right)},
\end{align*}
we have
\begin{align*}
\theta>0,\quad \dfrac{a+\theta}{q}-\theta>0\quad\mbox{and}\quad 1-\dfrac{a+\theta}{q}>0.
\end{align*}
According the Bishop-Gromov volume comparison theorem, applying H\"older's inequality, we have
\begin{align*}
\int_{B_{R}}f^{\theta}w^{-a}\leq&\left(\int_{B_{R}}fw^{1-q}\right)^{\theta}\left(\int_{B_{R}}w^{-q}\right)^{\frac{a+\theta}{q}-\theta}
\mathrm{Vol}\left(B_{R}\right)^{1-\frac{a+\theta}{q}}\\
\leq&C_{q,\theta,a} R^{(n-q)\theta+(n-q)\left(\frac{a+\theta}{q}-\theta\right)+n\left(1-\frac{a+\theta}{q}\right)}\\
=&C_{q,\theta,a}R^{n-a-\theta}.
\end{align*}
Notice that
\begin{align*}
n-a-\theta=&2-\dfrac{n-p}{(n-1)p}-\epsilon.
\end{align*}
We obtain
\begin{align*}
\int_{B_{R}} f^{\frac{(n-2)(p-1)}{(n-1)p}+\epsilon}w^{3-\frac2p-n}= &\int_{B_{R}}f^{\theta}w^{-a}\\
&\leq C_{a,\theta}R^{n-a-\theta}=C_{\epsilon}R^{2-\frac{n-p}{(n-1)p}-\epsilon}.
\end{align*}

The rest of the proof is similar to \autoref{thm:large-p}.
\end{proof}

\subsection{Rigidity results for solution with polynomial decay}

Next, we examine the solutions that, while they may not possess finite energy, exhibit a polynomial decay behavior at infinity. That is, we will give a proof of the third part of \autoref{thm:main1}.
\begin{theorem}\label{thm:decay}
Let $M^n$ be a complete, connected and noncompact Riemannian manifold of dimension $n$ and with nonnegative Ricci curvature. Let $u$ be a positive and weak solution to the critical $p$-Laplace equation \eqref{eq:critical-p} in $M^n$. Assume $1< p<n$ and
\begin{align*}
u(x)\leq Cr(x)^{-d},\quad\forall r(x)\geq C^{-1}
\end{align*}
for some positive numbers $C$ and $d$ with
\begin{align*}
d>\dfrac{(n-3p)(p-1)(n-p)}{p^2\left(n+2-3p\right)},
\end{align*}
where $r(x)=\mathrm{dist}(x,o)$ is the distance function from $x$ to some fixed point $o\in M^n$. Then $M^n$ is isometric to the Euclidean space $\mathbb{R}^n$ and the solution is give by
\begin{align*}
 u(x)=\left(a+b\abs{x-x_0}^{\frac{p}{p-1}}\right)^{-\frac{n-p}{p}}
\end{align*}
for some $x_0\in\mathbb{R}^n$ and $a, b>0$ satisfying
$$n\left(\dfrac{n-p}{p-1}\right)^{p-1}ab^{p-1}=1.$$
\end{theorem}

\begin{rem}We would like to the following three facts:
\begin{itemize}
\item Applying the weak comparison principle for weak positive $p$-superharmonic functions, we know that
\begin{align*}
 u(x)\geq C^{-1}r(x)^{-\frac{n-p}{p-1}},\quad\forall r(x)\geq1.
\end{align*}
Consequently,
\begin{align*}
 d\leq\dfrac{n-p}{p-1}.
\end{align*}

\item The case $d\geq\frac{n-p}{p}$ can be obtained from \autoref{thm:finite}. In fact, if $d>\frac{n-p}{p}$, then the Bishop-Gromov volume comparison theorem yields for all $R>1$
\begin{align*}
\int_{B_{R}}u^{\frac{pn}{n-p}}=&\int_{B_1}u^{\frac{pn}{n-p}}+\int_{B_R\setminus B_1}u^{\frac{pn}{n-p}}\\
\leq&C+C\int_1^Rr^{n-1-\frac{pnd}{n-p}}\dif r\\
\leq&C.
\end{align*}
In other words, $u\in L^{\frac{np}{n-p}}\left(M^n\right)$ and \autoref{thm:finite} can be utilized. If $d=\frac{n-p}{p}$, then for every positive number
$$\epsilon<\dfrac{(2n-1)p-n}{(n-1)p}$$
there holds true
\begin{align*}
q\coloneqq n-\epsilon>\dfrac{(n^2-3n+1)p+n}{(n-1)p}.
\end{align*}
Moreover,
\begin{align*}
0<n-\dfrac{pqd}{n-p}=n-q.
\end{align*}
Thus the Bishop-Gromov volume comparison theorem yields for all $R>1$
\begin{align*}
\int_{B_{R}}u^{\frac{pq}{n-p}}\leq C+C\int_{B_R\setminus B_1}r^{-\frac{pqd}{n-p}}\leq CR^{n-\frac{pqd}{n-p}}=CR^{n-q}.
\end{align*}
Specifically, \autoref{thm:finite} is applicable.
\item The case $p_n<p<n$ is trivial according to \autoref{thm:large-p}.
\end{itemize}
\end{rem}

\begin{proof}[Proof of \autoref{thm:decay}]
Denote by
\begin{align*}
w=\left(\dfrac{n-p}{p}\right)^{\frac{p-1}{p}}u^{-\frac{p}{n-p}} \quad\mbox{and}\quad f=\dfrac{n(p-1)}{p}w^{-1}\abs{\nabla w}^{p}+w^{-1}.
\end{align*}
According to \autoref{thm:basic-p}, it suffices to prove
\begin{align*}
\int_{B_{R}\setminus B_{R/2}}f^{2-\frac2p-\alpha}w^{3-\frac2p-n}\leq C_{\alpha}R^{2}
\end{align*}
for some real number
$$\alpha<\dfrac{n(p-1)}{(n-1)p}.$$

According to the classification results for large $p$ in \autoref{thm:large-p}, we may assume $1< p\leq p_n$ which implies
\begin{align*}
1<p\leq\dfrac{n^2}{3n-2}.
\end{align*}
This scenario arises exclusively when $n\geq3$. According to the local Cheng-Yau gradient estimate \eqref{eq:local-gradient}, we conclude that for every $x\in M^n$ and $R>0$
\begin{align*}
\sup_{B_{R}(x)}\abs{\nabla\ln u}\leq C\left(R^{-1}+\sup_{B_{4R}(x)}r^{-\frac{pd}{n-p}}\right),
\end{align*}
which gives
\begin{align*}
\sup_{B_{R}(x)}\abs{\nabla\ln w}\leq C\left(R^{-1}+\sup_{B_{4R}(x)}r^{-\frac{pd}{n-p}}\right).
\end{align*}
Recall $r(x)=\mathrm{dist}(x,o)$ is the distance function from the fixed point $o$ to $x$. Thus, if $r(x)\geq5$, then denote by $r(x)=5R$ and we obtain
\begin{align*}
\abs{\nabla\ln w}(x)\leq C\left(R^{-1}+r(y_x)^{-\frac{pd}{n-p}}\right)
\end{align*}
for some $y_x\in B_{4R}(x)$. By observing $R\leq r(y_x)\leq 9R$, we obtain
\begin{align*}
\abs{\nabla\ln w}\leq Cr^{-\frac{pd}{n-p}},\quad\forall r\geq 1.
\end{align*}
It is obvious
\begin{align*}
w^{-1}\leq Cr^{-\frac{pd}{n-p}}.
\end{align*}
We get
\begin{align*}
f\leq Cr^{-\frac{p^2d}{n-p}}w^{p-1}.
\end{align*}
We take
\begin{align*}
\alpha=\dfrac{p^2-n}{p(p-1)}<\dfrac{n(p-1)}{(n-1)p}
\end{align*}
to obtain for all $r\geq1$
\begin{align*}
f^{2-\frac2p-\alpha}w^{3-\frac2p-n}=&f^{\frac{n+2-3p}{p(p-1)}}fw^{-\frac{(n-3)p+2}{p}}\\
\leq &Cr^{-\frac{(n+2-3p)pd}{(p-1)(n-p)}}fw^{-\frac{n(p-1)}{p}}.
\end{align*}

According to \autoref{lem:crucial-p}, we obtain that, for each real number
$$0<\epsilon\leq\dfrac{(p-1)(n-p)}{p},$$
there holds
\begin{align*}
\int_{B_{R}}fw^{-\frac{n(p-1)}{p}+\epsilon}\leq C_{\epsilon}R^{n-\left(\frac{(n+1)p-n}{p}-\epsilon\right)}=C_{\epsilon}R^{\frac{n-p}{p}+\epsilon}.
\end{align*}
Hence, for every $R>1$ we have
\begin{align*}
R^{-2}\int_{B_{R}\setminus B_{R/2}}f^{2-\frac2p-\alpha}w^{3-\frac2p-n}\leq&CR^{-2}\int_{B_{R}\setminus B_{R/2}}r^{-\frac{(n+2-3p)pd}{(p-1)(n-p)}}fw^{-\frac{n(p-1)}{p}+\epsilon}w^{-\epsilon}\\
\leq &C_{\epsilon}R^{-2-\frac{(n+2-3p)pd}{(p-1)(n-p)}-\frac{pd\epsilon}{n-p}}\int_{B_{R}\setminus B_{R/2}}fw^{-\frac{n(p-1)}{p}+\epsilon}\\
\leq &C_{\epsilon}R^{-2-\frac{(n+2-3p)pd}{(p-1)(n-p)}-\frac{pd\epsilon}{n-p}+\frac{n-p}{p}+\epsilon}\\
= &C_{\epsilon}R^{\frac{n-3p}{p}-\frac{(n+2-3p)pd}{(p-1)(n-p)}+\left(1-\frac{pd}{n-p}\right)\epsilon}.
\end{align*}
Since
\begin{align*}
d>\dfrac{(n-3p)(p-1)(n-p)}{p^2(n+2-3p)},
\end{align*}
we can choose $\epsilon$ so small such that
\begin{align*}
\dfrac{n-3p}{p}-\dfrac{(n+2-3p)pd}{(p-1)(n-p)}+\left(1-\dfrac{pd}{n-p}\right)\epsilon<0
\end{align*}
and conclude that
\begin{align*}
R^{-2}\int_{B_{R}\setminus B_{R/2}}f^{2-\frac2p-\alpha}w^{3-\frac2p-n}\leq C.
\end{align*}
In view of \autoref{thm:basic-p}, we conclude that $f$ is a constant function and hence $M^n$ is isometric to the Euclidean space $\mathbb{R}^n$ and the solutions are given by the Aubin-Talenti bubbles.
\end{proof}

\subsection{Rigidity results for Liouville equations}

In the end of this section, we will give a proof of our second main \autoref{thm:main2}.

\begin{proof}[Proof of \autoref{thm:main2}]

Denote by
\begin{align*}
w=e^{-u} \quad\mbox{and}\quad f=(n-1)w^{-1}\abs{\nabla w}^n+w^{-1}.
\end{align*}
Using a similar argument as in the proof of \autoref{thm:main1}, it suffices to prove that $f$ is a constant function.

Firstly, we claim that for every $q<n$
\begin{align}\label{eq:crucial-n}
\int_{B_{R/2}}fw^{1-q}\leq C_qR^{-n}\int_{B_{R}}w^{n-q},\quad\forall R>0.
\end{align}
To see this, by observing
\begin{align*}
\Delta_nw=f.
\end{align*}
For every nonnegative function $\eta\in C_0^{\infty}\left(M^n\right)$,
\begin{align*}
&(n-1)\int \abs{\nabla w}^{n}w^{-q}\eta^{n}+\int w^{-q}\eta^{n}\\
=&\int fw^{1-q}\eta^{n}\\
=&\int\left(\Delta_nw\right)w^{1-q}\eta^{n}\\
=&(q-1)\int\abs{\nabla w}^nw^{-q}\eta^{n}-n\int\abs{\nabla w}^{n-2}w^{1-q}\hin{\nabla w}{\nabla\eta}\eta^{n-1}.
\end{align*}
Thus
\begin{align*}
\int w^{-q}\eta^{n}+\left(n-q\right)\int\abs{\nabla w}^nw^{-q}\eta^{n}\leq n\int\abs{\nabla w}^{n-1}w^{1-q}\abs{\nabla\eta}\eta^{n-1}.
\end{align*}
The H\"older inequality with the respective exponent pair $\left(n,\,\frac{n}{n-1}\right)$ yields
\begin{align*}
\int\abs{\nabla w}^{n-1}w^{1-q}\abs{\nabla\eta}\eta^{n-1}\leq\left(\int\abs{\nabla w}^nw^{-q}\eta^{n}\right)^{\frac{n-1}{n}}\left(\int w^{n-q}\abs{\nabla\eta}^n\right)^{\frac{1}{n}}.
\end{align*}
We obtain
\begin{align*}
\int w^{-q}\eta^{n}+\left(n-q\right)\int\abs{\nabla w}^nw^{-q}\eta^{n}\leq&n\left(\int\abs{\nabla w}^nw^{-q}\eta^{n}\right)^{\frac{n-1}{n}}\left(\int w^{n-q}\abs{\nabla\eta}^n\right)^{\frac{1}{n}},
\end{align*}
which implies
\begin{align*}
\int\abs{\nabla w}^nw^{-q}\eta^{n}\leq&C\gamma^n\int w^{n-q}\abs{\nabla\eta}^n,
\end{align*}
and
\begin{align*}
\int w^{-q}\eta^{n}\leq C\int w^{n-q}\abs{\nabla\eta}^n.
\end{align*}
Hence, for every $q<n$
\begin{align*}
\int fw^{1-q}\eta^{n}\leq C \int w^{n-q}\abs{\nabla\eta}^n.
\end{align*}
Consequently, we obtain the desired integral inequality \eqref{eq:crucial-n}.

Notice that
$$n+\dfrac2n-2<n$$
since $n\geq 2$. Denote by
\begin{align*}
q=n+\dfrac2n-2.
\end{align*}
We obtain from \eqref{eq:critical-n}
\begin{align*}
\int_{B_{R/2}}fw^{3-n-\frac2n}=&\int_{B_{R/2}}fw^{1-q}\\
\leq &CR^{-n}\int_{B_{R}}w^{n-q}\\
= &CR^{-n}\int_{B_{R}}w^{2-\frac2n}.
\end{align*}
That is
\begin{align*}
\int_{B_{R/2}}fw^{3-n-\frac2n}\leq CR^{-n}\int_{B_{R}}w^{2-\frac2n}.
\end{align*}
By the assumption
\begin{align*}
w(x)\leq Cr(x)^{\frac{n}{n-1}}F^{\frac{n}{2(n-1)}}(r(x)),\quad\forall r(x)>0.
\end{align*}
We have
\begin{align*}
\int_{B_{R/2}}fw^{3-n-\frac2n}\leq&CR^{-n}R^{2}F(R)\mathrm{Vol}\left(B_{R}\right).
\end{align*}
The Bishop-Gromov volume comparison theorem yields
\begin{align}\label{eq:quadratic-n}
\int_{B_{R/2}}fw^{3-n-\frac2n}\leq CR^{2}F(R).
\end{align}

According to \autoref{thm:basic-n}, we know that if $f$ is not a constant function, then for every $\alpha<1$ there holds true
\begin{align}\label{eq:quadratic-n1}
\limsup_{R\to\infty}\dfrac{1}{R^2G(R)}\int_{B_{R}}f^{-\alpha}w^{1-n}\abs{\nabla w}^{2n-2}=\infty
\end{align}
where $G(s)=F(2s)$ satisfying
\begin{align*}
\int^{\infty}\dfrac{\dif s}{sG(s)}=\infty.
\end{align*}
Since
\begin{align*}
 fw\geq(n-1)\abs{\nabla w}^{n},
\end{align*}
we take
$$\alpha=1-\dfrac2n<1$$
to obtain
\begin{align*}
f^{-\alpha}w^{1-n}\abs{\nabla w}^{2n-2}\leq Cf^{2-\frac2n-\alpha}w^{3-n-\frac2n}=Cfw^{3-n-\frac2n}.
\end{align*}
It follows from \eqref{eq:quadratic-n} and \eqref{eq:quadratic-n1}
\begin{align*}
\infty=\limsup_{R\to\infty}\dfrac{1}{R^2F(R)}\int_{B_{R/2}}fw^{3-n-\frac2n}\leq C,
\end{align*}
which is a contradiction. Hence $f$ is a constant function.

We can proceed with the proof in the following manner. Using a similar argument as in the proof of \autoref{thm:main1}, we conclude that $M^n$ is isometric to the Euclidean space $\mathbb{R}^n$ and the function $h=\frac12\abs{\mathbf{W}}^2$ satisfies
\begin{align*}
 \nabla^2h=c^2\mathrm{Id} \quad\mbox{and}\quad\nabla h=c\mathbf{W}
\end{align*}
for some positive number $c$. Immediately, it follows
\begin{align*}
h(x)=h(x_0)+\dfrac{c^2}{2}\abs{x-x_0}^2
\end{align*}
for some $x_0\in\mathbb{R}^n$. Thus
\begin{align*}
\abs{\nabla w}^{n-2}\nabla w=\mathbf{W}=c\left(x-x_0\right)
\end{align*}
which implies
\begin{align*}
\nabla w=c^{\frac{1}{n-1}}\abs{x-x_0}^{\frac{1}{n-1}-1}\left(x-x_0\right).
\end{align*}
We obtain
\begin{align*}
w=w(x_0)+\dfrac{n-1}{n}c^{\frac{1}{n-1}}\abs{x-x_0}^{\frac{n}{n-1}}.
\end{align*}
Direct computation yields
$$\Delta_nw=nc$$
and
\begin{align*}
f &=(n-1)w^{-1}\abs{\nabla w}^n+w^{-1}\\
&=\left(1+(n-1)c^{\frac{n}{n-1}}\abs{x-x_0}^{\frac{n}{n-1}}\right)\left(w(x_0)+\dfrac{n-1}{n}c^{\frac{1}{n-1}}\abs{x-x_0}^{\frac{n}{n-1}}\right)^{-1}.
\end{align*}
The equation $\Delta_n w=f$ then yields
\begin{align*}
 w(x_0)=\dfrac{1}{nc}.
\end{align*}
Therefore,
\begin{align*}
 u(x)=-\ln\left(\dfrac{1}{nc}+\dfrac{n-1}{n}c^{\frac{1}{n-1}}\abs{x-x_0}^{\frac{n}{n-1}}\right)
\end{align*}
and we complete the proof.
\end{proof}

\vspace{2ex}


\end{document}